\documentclass[11pt]{amsart}

\usepackage{amsmath,amsthm,amsfonts}
\usepackage{graphicx}

\usepackage[hmargin=2.5cm,vmargin=2.5cm]{geometry}

\newcommand{\Sym}[1]{{{\mathfrak{S}_{#1}}}}
\newcommand{\Class}[1]{{{\mathcal{C}_{#1}}}}

\newtheorem{thm}{Theorem}[section]
\newtheorem{lem}[thm]{Lemma}

\newtheorem{cor}[thm]{Corollary}
\newtheorem{conjecture}[thm]{Conjecture}

\theoremstyle{definition}
\newtheorem{remark}[thm]{Remark}

\newtheorem{example}[thm]{Example}

\newcommand{\p}{{\pi}}
\renewcommand{\a}{{\alpha}}
\renewcommand{\b}{{\beta}}
\newcommand{\s}{{\sigma}}

\newcommand{\Q}{\mathbb{Q}}
\newcommand{\C}{\mathbb{C}}

\newcommand{\pr}[1]{\left( {#1} \right)}
\newcommand{\map}[3]{{#1} \! : \! {#2} \! \longrightarrow \! {#3}}

\newcommand{\len}[1]{\ell({#1})}
\newcommand{\hsf}[1]{{h_{#1}}}
\newcommand{\esf}[1]{{e_{#1}}}
\newcommand{\ssf}[1]{{s_{#1}}}

\newcommand{\mathscr}[1]{\mathcal{#1}}

\newcommand{\ocfacts}[2]{c_{#1\,;\,#2}}
\newcommand{\ocGS}[1]{\Psi_{#1}}
\newcommand{\okGS}[2]{\Psi_{#1,#2}}

\newcommand{\mgfacts}[2][g]{M_{#2, #1}}
\newcommand{\igfacts}[3][g]{\widetilde{C}_{{#2},{#1}}(#3)}
\newcommand{\pgfacts}[3][g]{P_{{#2},{#1}}(#3)}
\newcommand{\cgfacts}[3][g]{F_{{#2},{#1}}(#3)}

\newcommand{\icfacts}[2]{\tilde{c}_{#1\,;\,#2}}
\newcommand{\pGS}[1]{\widehat{\Psi}_{#1}}
\newcommand{\icGS}[1]{\widetilde{\Psi}_{#1}}
\newcommand{\ikGS}[2]{\widetilde{\Psi}_{#1,#2}}

\newcommand{\class}[1]{\mathscr{C}_{#1}}

\newcommand{\depth}[1]{\langle #1 \rangle}
\newcommand{\target}[1]{\Pi\left(#1\right)}

\newcommand{\fmap}[1]{\mathscr{M}_{#1}}
\newcommand{\fgraph}[1]{\mathscr{G}_{#1}}

\newcommand{\bq}{{\mathbf{q}}}
\newcommand{\bx}{{\mathbf{x}}}

\newcommand{\xis}{\phi}
\newcommand{\QD}{P}
\newcommand{\QS}{S}
\newcommand{\QQ}{Q}
\newcommand{\ve}{\nu_{\mathrm{e}}}
\newcommand{\vo}{\nu_{\mathrm{o}}}
\newcommand{\pd}[2]{\frac{\partial #1}{\partial #2}}

\newcommand{\Dop}[1]{D_{#1}}
\newcommand{\Dopx}{\Dop{\bx}}

\begin{document}


\title{Inequivalent factorizations of permutations}

\author{G.~Berkolaiko}
\address{Department of Mathematics, Texas A\&M University, College Station, TX, USA}
\author{J.~Irving}
\address{Department of Mathematics \& Computing Science, Saint Mary's University, Halifax, NS, Canada}

\begin{abstract}
Two factorizations of a permutation into products of cycles are equivalent if one can be obtained from the other by repeatedly interchanging adjacent disjoint factors. This paper studies the enumeration of equivalence classes under this relation. 

We establish general connections between inequivalent factorizations and other well-studied classes of permutation factorizations, such as monotone factorizations.  We also obtain several specific enumerative results, including closed form generating series for inequivalent minimal transitive factorizations of permutations having up to three cycles. Our derivations rely on a new correspondence between inequivalent factorizations and acyclic alternating digraphs.   Strong similarities between the enumerative results derived here and analogous ones for ``ordinary'' factorizations suggest that a unified theory remains to be discovered. 
\end{abstract}


\maketitle

\section{Introduction}

\subsection{Notation}

We adhere to standard notation and terminology concerning permutations.  We
write $\Sym{n}$ for the symmetric group on the symbols
$\{1,2,\ldots,n\}$, and we multiply permutations from right to
left. The number of cycles in $\p \in \Sym{n}$ is denoted by
$\len{\p}$.  For a composition $\a=(\a_1,\ldots,\a_m)$ of $n$, we write
$\class{\a}$ for the \emph{conjugacy class} of $\Sym{n}$ consisting of
all permutations whose disjoint cycles are of lengths
$\a_1,\ldots,\a_m$. Elements of $\class{\a}$ are said to be of
\emph{cycle type} $\a$. Permutations of cycle type $(k,1,1,\ldots,1)$
are called \emph{$k$-cycles}, with $2$-cycles more commonly referred
to as \emph{transpositions}.  We typically suppress cycles
of length 1 when writing permutations in disjoint cycle notation. Thus
$(i\,j)$ denotes a transposition in $\Sym{n}$, with the value of $n$
to be understood from context.

For any list of integers $\b = (\b_1,\b_2,\ldots)$ with finite support, let $|\b| = \sum_k \b_k$ and let $\len{\b}$ be the number of nonzero entries of $\b$. In particular, for $\p \in \Class{\a} \subset \Sym{n}$ we have $|\a|=n$ and $\len{\a} = \len{\p}$.

For an integer partition $\l$ and a set of indeterminates $\bx=(x_1,\ldots,x_n)$, we write $\hsf{\l}(\bx)$, $\esf{\l}(\bx)$ and $\ssf{\l}(\bx)$, respectively, for the complete, elementary, and Schur symmetric polynomials indexed by $\l$.  We adopt the convention that each of these polynomials is 0 when $\l$ is not a partition.

The ring of formal power series in indeterminates $\bx=(x_1,\ldots,x_m)$ over the ring $R$ is denoted by $R[[\bx]]$.  If $f \in R[[\mathbf{x}]]$ and $\mathbf{i}=(i_1,\ldots,i_m)$ is a list of nonnegative integers, then we write
$[\bx^{\mathbf{i}}]\,f$ for the coefficient of the monomial $\bx^{\mathbf{i}} = x_1^{i_1} \cdots x_m^{i_m}$ in $f$.  We let $\Dopx$ denote the total derivative operator on $R[[\mathbf{x}]]$, namely
$\Dopx = \sum_{i=1}^m x_i \pd{}{x_i}$.

\subsection{Factorizations of Permutations}
A \emph{factorization} of a permutation  $\p \in \Sym{n}$ is a tuple $f=(\s_1,\ldots,\s_r)$ where each $\s_i \in \Sym{n}$ and $\p = \s_1 \cdots \s_r$.  The $\s_i$ are the \emph{factors} of $f$. The number of factors, $r$, is the \emph{length} of $f$, and is denoted by $\len{f}$.  We will generally be less formal and
write a factorization simply as the product of its factors.  For instance, 
\begin{equation}
\label{factexample}
(1\,2\,3)(4\,6)\cdot (2\,4\,6\,5)\cdot (1\,4)(2\,3)(5\,6)
\end{equation}
is a factorization of $(1\,4\,2)(3\,6)(5)$ of length $3$.

Let $f$ be a factorization  of $\p \in \Sym{n}$.  We define the \emph{class} of $f$ to be the cycle type of $\p$, while the \emph{signature} of $f$ is
the list $\b=(\b_2,\b_3,\ldots)$, where $\b_k$ is the total number of $k$-cycles amongst all factors. The \emph{depth} of $f$, denoted by $\depth{f}$, is defined as
\begin{equation*}
 \depth{f}=\sum_{j\ge 2} (j-1)\b_j.
\end{equation*}  
This is the minimum total number of transpositions required to decompose all factors of $f$.   Note that the depth of a factorization increases when an additional factor is inserted, except in the case where the extra factor is the identity.  
The factorization~\eqref{factexample}, above, is of class $(3,2,1)$, signature $(4,1,1,0,\ldots)$ and depth $9$.

A factorization in $\Sym{n}$ is {\em transitive} if the group generated by its factors acts transitively on $\{1,2,\ldots,n\}$. For instance, ~\eqref{factexample} is a  transitive factorization in $\Sym{6}$, whereas
\begin{equation*}
(1\,3\,2)(5\,6)\cdot (2\,4)(1\,3) \cdot (1\,4)(5\,6)
\end{equation*}
is not because $\{5,6\}$ is an invariant subset.  It is not difficult to show that for every transitive factorization $f$ of $\p \in \Sym{n}$ there is a unique nonnegative integer $g$ such that
\begin{equation}
  \label{eq:rank} 
  \depth{f} = n+\len{\p}-2+2g.
\end{equation} 
This $g$ is called the \emph{genus} of $f$.  A factorization is of genus 0 precisely when it is transitive and has minimal depth among all factorizations of the same class.  For this reason, genus 0 factorizations are said to be \emph{minimal transitive}.


Permutation factorizations have been studied for a long time in various  guises.
From an algebraic point of view, every question regarding factorizations is a question about the  structure of the symmetric group, and there is a well-trodden bridge between factorizations and the representation theory of $\Sym{n}$. (See, for instance,~\cite{BedGou92,GouJac94,Jac88,Stan81}.)
Factorizations also have a geometric flavour, in the sense that they encode cellular decompositions of surfaces --- that is, maps. (See~\cite{BouSch_aam00, GouSch98,MorVas13} and references therein.)
The notions of transitivity and genus of factorizations arise naturally from the geometric point of view, being equivalent to connectedness and genus of the associated maps. 



\subsection{Cycle Factorizations}
\label{sec:cyclefacts}

This paper is primarily concerned with the combinatorics of \emph{cycle factorizations}, which are factorizations in which every factor is a cycle of some length. For example,
$$
	(1\,3)\cdot(2\,4\,5)\cdot(1\,2\,3)\cdot(2\,5)\cdot(3\,6)\cdot(1\,2\,3\,4)
$$
is a cycle factorization with signature $(3,2,1,0,\ldots)$.   A \emph{$k$-cycle factorization} is a factorization in which all factors are $k$-cycles.

The study of $2$-cycle factorizations (\emph{i.e.} factorizations into transpositions) dates back at least to Hurwitz, who used them to encode topologically inequivalent branched coverings of the sphere. Hurwitz found~\cite{Hur_ma91} the following beautiful formula for the number of minimal transitive 2-cycle factorizations of any permutation $\pi \in \Sym{n}$ of cycle type $\a=(\a_1,\ldots,\a_m)$:
\begin{equation}
\label{eq:hurwitz}
		 n^{m-3} (n+m-2)! \prod_{i=1}^m \frac{\a_i^{\a_i}}{(\a_i-1)!}.
\end{equation}
See~\cite{BouSch_aam00,GouJac97} for modern derivations of Hurwitz's formula and~\cite{Str96} for a reconstruction of Hurwitz's original proof.
More recently, the celebrated ELSV formula has extended this geometric connection to link the combinatorics of 2-cycle factorizations with the intersection theory of moduli spaces of curves; see~\cite{ELSV01,GouJacVak01,GouJacVak05,OkoPan09}.

Taking $\a=(n)$ in Hurwitz's formula shows that there are $n^{n-2}$ factorizations of the full cycle $(1\,2\,\cdots\,n)$ into $n-1$ transpositions.\footnote{All factorizations of the full cycle are necessarily transitive.}  This famous result is often attributed to D\'enes~\cite{Den59}, who proved it using a correspondence with labelled trees.  More generally, it is known~\cite{Irv_cjm09,Spr96} that there are 
\begin{equation}
\label{eq:fullcyclefacts}
		\frac{n^{\ell-1} \ell!}{\prod_{k} \b_k!}
\end{equation}
minimal transitive factorizations of $(1\,2\,\cdots\,n)$ with signature $\b=(\b_2,\b_3,\ldots)$ and length $\ell=|\b|$.

Succinct counting formulae such as~\eqref{eq:hurwitz} and~\eqref{eq:fullcyclefacts} do not exist for any other classes of cycle factorizations, even in the minimal transitive case.  Nonetheless, there is evidence to suggest these factorizations have a rich combinatorial structure.  (See Section~\ref{sec:survey_connections}.)

\subsection{Inequivalent Factorizations}
\label{sec:inequivalent_facts}

There is a natural equivalence relation $\sim$ on the set of cycle factorizations,  defined by stipulating that two such factorizations are equivalent if one can be obtained
from the other by iteratively swapping adjacent, disjoint
(and hence commuting) factors.  For example,
\begin{equation*}
  (3\,4\,5)\cdot(1\,2)\cdot(2\,3\,5)\cdot(1\,4) \sim (1\,2)\cdot(3\,4\,5)\cdot (1\,4) \cdot(2\,3\,5).
\end{equation*}
Although this relation can be extended to arbitrary factorizations in an obvious way, we emphasize that we have defined it only for cycle factorizations.

The principal focus of this paper is the combinatorics of equivalence classes of cycle factorizations under this relation. As such, we shall abuse terminology and henceforth refer to the class containing such a factorization $f$ simply as the \emph{inequivalent factorization $f$}. Note that length, class, signature and depth are invariant under commutation of disjoint adjacent factors, so it is sensible to apply these terms to inequivalent factorizations. 

Let us write $\icfacts{\a}{\b}$ for the number of minimal transitive inequivalent factorizations with signature $\b$ of any permutation of cycle type $\a$. It is convenient to define, for each $m \geq 1$, the generating series
\begin{align}
\label{eq:icGS}
		\icGS{m}(\bx,\bq) =
	 		\sum_{\a, \b} \icfacts{\a}{\b} \frac{x_1^{\a_1} \cdots x_m^{\a_m}}{\a_1 \cdots \a_m} \bq^{\b},
\end{align}
where the sum extends over all $m$-part compositions $\a=(\a_1,\ldots,\a_m)$ and all finitely supported lists $\b = (\b_2,\b_3,\ldots)$ of nonnegative integers. Throughout this article, the indeterminate $q_k$ is a marker for $k$-cycles and $\bq=(q_2,q_3,\ldots)$.  Let $\ikGS{m}{k}$ be the restriction of $\icGS{m}$ to $k$-cycle factorizations, obtained  by setting $q_k=1$ and $q_i=0$ for $i \neq k$.  

In comparison with their ``ordinary'' analogues,  little is known about inequivalent factorizations, and all specific enumerative results are restricted to the minimal transitive case.  The first such results were obtained by Eidswick~\cite{Eid89} and Longyear~\cite{Lon89}, who independently showed that there are
$$
	\icfacts{(n)}{(n-1)}=\frac{1}{n-1}\binom{3n-3}{n-2}
$$
inequivalent factorizations of the full cycle $(1\,2\,\cdots\,n)$ into $n-1$ transpositions. (Note that these are necessarily minimal transitive.)  Longyear's approach involved commuting factorizations into canonical forms, leading to the functional equation
\begin{equation}
\label{eq:longyear}
	h = 1 + xh^3
\end{equation}
for the series $h(x) = \pd{}{x} \ikGS{1}{2}(x)$.  This result was extended to $k$-cycle factorizations by Goulden and Jackson~\cite{GouJac94}, who obtained  $\ikGS{1}{k}$ as a corollary of their work on Macdonald's $u_\lambda$ symmetric functions. Springer~\cite{Spr96} generalized Longyear's canonical form and used a correspondence with trees to derive the following analogue of~\eqref{eq:fullcyclefacts} for inequivalent factorizations of the full cycle:
\begin{equation}
\label{eq:springer}
	\icfacts{(n)}{\b} = \frac{(2n+\ell-2)!}{(2n-1)! \prod_k \b_k!}.
\end{equation}

Inequivalent factorizations of permutations other than the full cycle were first studied by Goulden, Jackson and Latour \cite{GouJacLat01}, who showed that
\begin{equation}
\label{eq:gouldenjackson}
	\ikGS{2}{2}(x_1,x_2) = \log\pr{1+x_1 x_2 h(x_1) h(x_2) \frac{h(x_1)-h(x_2)}{x_1-x_2}},
\end{equation}
where $h$ is defined by~\eqref{eq:longyear}.
Their derivation again employs commutation to canonical form, but also relies on a somewhat intricate inclusion-exclusion argument.  Although not stated in~\cite{GouJacLat01}, it is possible to extract coefficients from this series to obtain the following ``inequivalent'' analogue of~\eqref{eq:hurwitz} in the case where $\a$ has two parts (see Section~\ref{sec:transpositions} for details):
\begin{equation}
\label{eq:coeffextract}
		\icfacts{(n,m)}{(n+m)}=\frac{2nm}{n+m} \sum_{k \geq 0} \binom{3n}{n-1-k}\binom{3m}{m-1-k}.
\end{equation}

Springer's formula~\eqref{eq:springer} was proved again in~\cite{BerHarNov_prep08,Irv_cjm09}, where a simple functional equation for $\icGS{1}$ was derived from graphical models for inequivalent factorizations closely related to those employed in this article (see Theorem~\ref{thm:digraphs}). In~\cite{Irv_cjm09} this approach was also used to yield a compact expression for $\icGS{2}$, generalizing the Goulden-Jackson-Latour series~\eqref{eq:gouldenjackson}.  These results will be restated in an alternative form and reproved below (see Theorem~\ref{thm:Q}).
 
Although inequivalent factorizations were initially studied as a
combinatorial curiosity, we will witness
surprisingly close structural ties between them and their ``ordinary''
cousins.  Recently, inequivalent factorizations have also appeared in
the physics literature in connection to quantum chaotic transport (see
Section~\ref{sec:physics}).

\section{Survey of Results}
\label{sec:survey}


Our results on inequivalent factorizations can be separated into two distinct, but not wholly disjoint, categories:   (1) general relationships with other  classes of factorizations, and (2) specific enumerative results.

The bulk of our technical effort has been dedicated to the enumeration of minimal transitive inequivalent factorizations. While we have substantially extended all previous work along these lines, we believe the relationships we have uncovered between various classes of factorizations (both proven and conjectured) are of greater fundamental interest than our specific enumerative results.  As such, we have organized the article to emphasize these connections.

In this section we present a high level summary of our work, deferring various technical details until later.  We hope this affords the reader a glimpse at the grand structure of transitive factorizations.

\subsection{Connections with Other Classes of Factorizations}
\label{sec:survey_connections}

For a composition $\a=(\a_1,\ldots,\a_m)$ of $n$, let
$\cgfacts{\a}{r}$ denote the number of (transitive) genus $g$
factorizations of length $r$ of any $\p \in \Class{\a}$. Let
$\pgfacts{\a}{r}$ be the number of these which are \emph{proper}, by
which we mean that no factor is the identity permutation.  We stress
that $\cgfacts{\a}{r}$ and $\pgfacts{\a}{r}$ count factorizations into
permutations of arbitrary cycle structure, as opposed to inequivalent
factorizations which are cycle factorizations.

Since the removal of identity factors does not alter depth or transitivity, every genus $g$ factorization of class $\a$ can be built by inserting identity factors into a unique proper factorization of the same class and genus. In this way we obtain
\begin{equation}
\label{eq:polynomial}
	\cgfacts{\a}{r} = \sum_k \binom{r}{k} \pgfacts{\a}{k}
\end{equation}
for all nonnegative integers $r$.  Since each factor of a proper factorization contributes at least 1 to depth, \eqref{eq:rank} implies that $\pgfacts{\a}{k} = 0$ whenever $k > n+m-2+2g$.  Thus the right-hand side of~\eqref{eq:polynomial} is polynomial in $r$. We therefore extend the definition of $\cgfacts{\a}{r}$ to all values of $r$ by identifying it with this polynomial.

Transitive factorizations of specified length have been studied only in genus 0, in which case they correspond with a class of planar maps known as \emph{constellations}.  Bousquet-Melou and Schaeffer have counted constellations via an ingenious bijective decomposition  into decorated trees,  showing in~\cite{BouSch_aam00} that\footnote{  
Hurwitz's formula~\eqref{eq:hurwitz} can be obtained as an ``extremal'' case of~\eqref{eq:constellation}. See~\cite{BouSch_aam00} for details.} 
  \begin{equation}
    \label{eq:constellation}
    \cgfacts[0]{\alpha}{r} = r ((r-1)n-1)_{(m-3)} \prod_{i =1}^m \a_i \binom{r\a_i-1}{\a_i},
    \qquad r \geq 2,
  \end{equation}
where  $x_{(k)}=x(x-1)(x-2)\cdots(x-k+1)$ for $k \geq 0$ and $x_{(-k)}=1/(x+k)_{(k)}$. 
Since both sides are polynomial in $r$ and equality holds for all $r \geq 2$, we conclude that~\eqref{eq:constellation} is a polynomial identity.


A \emph{monotone factorization} (also called a \emph{primitive}
factorization) is a 2-cycle factorization whose factors weakly
increase from left to right with respect to greatest element. That is,
factorization $(a_1\,b_1)\cdot(a_2\,b_2)\cdots(a_r\,b_r) $ is monotone when $1 \leq a_i
< b_i \leq n$ for all $i$ and $b_1 \leq b_2 \leq \cdots \leq b_r$.  For instance,
$$
	(2\,3)\cdot (3\,4)\cdot (1\,4)\cdot (3\,4)\cdot (4\,5)	
$$
is a minimal transitive monotone factorization of $(1\,2\,3)(4\,5)$.  Monotone factorizations of the full cycle were initially studied by Gewurz and Merola~\cite{GewMer06}, who showed they are counted by the Catalan numbers.  Matsumoto and Novak~\cite{MatNov_fpsac10,MatNov_imrn13} later initiated a more general study in connection with the
expansion of certain matrix integrals.  A
thorough ``cut-join'' analysis was given in~\cite{GouGuaNov_cjm13}, where the structure of transitive monotone factorizations was shown to closely parallel the that of general  2-cycle factorizations.  

For any $\p \in \Class{\a}$, let $\mgfacts{\a}$ be the number of genus $g$ monotone factorizations of $\p$, and let $\igfacts{\a}{r}$ be the number of inequivalent genus $g$ factorizations of $\p$ of length $r$. Then we have the following relationship between factorizations of fixed genus and their proper, inequivalent, and monotone variants.

\begin{thm}
  \label{thm:connections}
  For any composition $\a$ and any $g \geq 0$,
  \begin{align*}
    \label{eq:connections}
        (-1)^{|\a|+\len{\a}} \,\mgfacts{\alpha} 
        = \sum_{r \geq 0} (-1)^{r} \igfacts{\alpha}{r}
         	= \sum_{r \geq 0} (-1)^{r} \pgfacts{\alpha}{r} = \cgfacts{\alpha}{-1}.
  \end{align*}
\end{thm}

Note that the rightmost equality in Theorem~\ref{thm:connections} is
obtained simply by evaluating~\eqref{eq:polynomial} at $r=-1$.  Thus
the true content of the theorem is the other equalities, which will be
established in Section~\ref{sec:connections} as consequences of
somewhat more general results.  In particular,
Theorem~\ref{thm:propercorrespondence} describes the connection
between proper and inequivalent factorizations (which comes by way of
the Cartier-Foata commutation monoid and which remains valid when
controlling for the signature $\beta$ of the factorization) and
Theorem~\ref{thm:primcorrespondence} provides the link between
monotone and inequivalent factorizations (for which we provide both
combinatorial and algebraic proofs).

Interestingly, the relationship between $\mgfacts{\alpha}$ and $\pgfacts{\alpha}{r}$ can also be deduced by comparing the work of Matsumoto/Novak~\cite{MatNov_fpsac10,MatNov_imrn13} and Collins~\cite[Theorem
2.4]{Col_imrn03}, where enumerations of monotone and proper
factorizations, respectively, appear in the asymptotic expansion of
integrals over the unitary group.  We also note that, while the appearance of $\cgfacts{\alpha}{-1}$ in Theorem~\ref{thm:connections} is reminiscent of Stanley's evaluation of the chromatic polynomial (to count acyclic orientations), we are not aware of any combinatorial meaning of $\cgfacts{\alpha}{-k}$ in general.

From Theorem~\ref{thm:connections} we can immediately recover the following beautiful counting formula for minimal transitive monotone factorizations, originally due to Goulden, Guay-Paquet and Novak ~\cite{GouGuaNov_cjm13}.   (During the preparation of this article we discovered that Chapuy~\cite{Chapuy_private13} has independently arrived at this result in essentially the same manner.)

\begin{cor}[{\cite[Theorem 1.1]{GouGuaNov_cjm13}}]
  \label{thm:goulden}
  For any composition $\a=(\a_1,\ldots,\a_m)$ of $n$, we have
  \begin{equation*}
\mgfacts[0]{\a} = (2n+1)^{(m-3)} \prod_{i=1}^m \a_i \binom{2\a_i}{\a_i},
  \end{equation*}
  where $x^{(k)} = x(x+1)\cdots(x+k-1)$ and $x^{(-k)} = 1/(x-k)^{(k)}$
  for nonnegative integers $k$.
\end{cor}

\begin{proof}
Take $g=0$ in Theorem~\ref{thm:connections} and set $r=-1$ in~\eqref{eq:constellation} to evaluate  $\cgfacts[0]{\alpha}{-1}$.
\end{proof}

In the case $\a=(n)$,  Corollary~\ref{thm:goulden} yields the Catalan number $\mgfacts[0]{(n)} = \frac{1}{n}\binom{2n-2}{n-1}$, in accordance with Gewurz and Merola's early result~\cite{GewMer06}.

\subsection{Minimal Transitive Inequivalent Factorizations}
\label{sec:enumerative_results}

Central to our study of inequivalent factorizations  is a new graphical model of them as \emph{alternating  maps}. These are embeddings of  directed graphs in orientable surfaces such that the edges encountered on a cyclic tour around any vertex alternate in direction.  
\begin{figure}[t]
  \centering
  \includegraphics[width=0.45\textwidth]{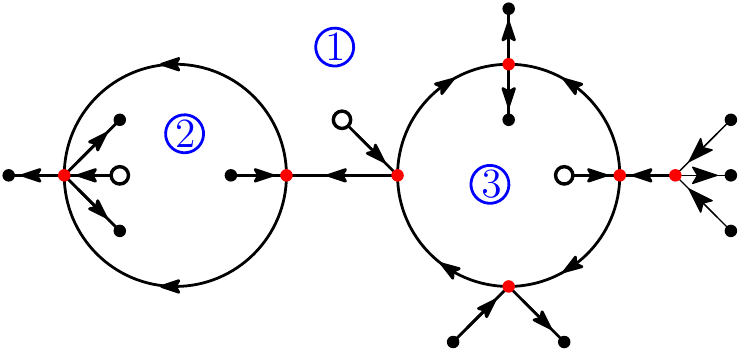}
  \caption{An alternating map satisfying the conditions of Theorem~\ref{thm:digraphs}, with $g=0$, $m=3$, $(\a_1,\a_2,\a_3)=(4,2,1)$ and $(\b_2,\b_3,\ldots)
  =(6,1,0,\ldots)$.}
  \label{fig:mapexample}
\end{figure}
Clearly every vertex of an alternating map is either a leaf or has even total degree, and the out-directed and in-directed leaves are precisely the sources and sinks, respectively.  Fig.~\ref{fig:mapexample} shows a preliminary example of the following correspondence, which will be established in Section~\ref{sec:models}.  
\begin{thm}
  \label{thm:digraphs}
  Let $\a=(\a_1,\ldots,\a_m)$ be an $m$-part composition of $n$ and let $\p \in \Sym{n}$ be of cycle type $\a$.
  Inequivalent genus $g$ factorizations of $\p$
  with signature $\b=(\b_2,\b_3,\ldots)$
  are in one-to-one correspondence with acyclic alternating genus $g$ maps with
  $m$ labelled faces in which 
  \begin{enumerate}
  \item[(a)] every vertex is a source, a sink, or has even total degree $\geq 4$,
  \item[(b)] face $i$ contains $\a_i$ sources and $\a_i$ sinks, with one source distinguished, and
  \item[(c)] there are  $\b_k$ vertices of degree $2k$, for $k \geq 2$.
  \end{enumerate}
\end{thm}


In Section~\ref{sec:proof} we employ this correspondence to give compact expressions for the minimal transitive generating series $\icGS{m}$ in cases $m =1,2,3$.  (See Theorem~\ref{thm:Q}, below.)  The real novelty here is our expression for $\icGS{3}$, since $\icGS{1}$ and $\icGS{2}$ have been found previously in different but equivalent forms~\cite{BerHarNov_prep08,Irv_cjm09,Spr96}. We have restated these  results  for completeness and unification. Both will be proved in Section~\ref{sec:proof} as introductory examples of our methods.

\begin{thm}
\label{thm:Q}
	Let  $\xis \in \Q[\bq][[x]]$ be the unique solution of
	\begin{align}
	\label{eq:Q_xis}
		\xis = x(1-\QQ(\xis))^{-2},
	\end{align}
	where $\QQ$ is defined by
        \begin{equation}
        \label{eq:Q}
          Q(z) = \sum_{k \geq 2} q_k z^{k-1}.
        \end{equation}
	Set
	\begin{align*}
		\QS(\xis) &= 1-\QQ(\xis)  \\
		\QD(\xis) &= 1-\QQ(\xis)-2 \xis \QQ'(\xis).
	\end{align*}
 Then we have
	\begin{align}
		\Dopx \icGS{1}
					&= \xis_1 \QS_1 \label{eq:Qm1} \\
		\Dopx \icGS{2}
					&= \xis_1\xis_2 \pr{\frac{\QS_1}{\QD_1}-\frac{\QS_2}{\QD_2}}  \frac{\QQ_1-\QQ_2}{(\xis_1-\xis_2)(\xis_1\QS_1 - \xis_2\QS_2) } 		
		\label{eq:Qm2}\\
		\icGS{3} 
			&= 2\xis_1\xis_2\xis_3\sum_{i =1}^3 \frac{1}{\QD_i} \prod_{j \neq i} 
		\frac{\QQ_i-\QQ_j}{(\xis_i-\xis_j)(\xis_i\QS_i-\xis_j\QS_j)} \label{eq:Qm3},
	\end{align}
	where 	$\xis_i = \xis(x_i,\bq)$, $\QQ_i = Q(\xis_i)$, $\QS_i = \QS(\xis_i)$  and $\QD_i=\QD(\xis_i)$. 
	\end{thm}


The similarity between expressions~\eqref{eq:Qm1}--\eqref{eq:Qm3} is suggestive of a common form for $\icGS{m}$, valid for all $m$. Finding such a form remains a topic for future study, as our methods become impractical for $m \geq 4$ and the \emph{ad hoc} nature of our derivations in cases $m=1,2,3$ sheds little light on the matter.  We also note that although $\icGS{3}$ is obviously symmetric in the $\xis_i$, we have found no symmetric function expansion that is remotely as concise as the ``alternating'' form given above.  

The presence of $\Dopx$ in~\eqref{eq:Qm1} and~\eqref{eq:Qm2}, and its absence in~\eqref{eq:Qm3}, is not well understood combinatorially. However,  we shall see below that this situation closely  parallels known results regarding ordinary cycle factorizations.

Theorem~\ref{thm:Q} can be specialized to obtain generating series for
inequivalent $k$-cycle factorizations. The restriction to
$\ikGS{1}{k}$ and $\ikGS{2}{k}$ is routine, while the evaluation of
$\ikGS{3}{k}$  rests on technical lemmas which we have relegated to the appendix. The results are given in Section~\ref{sec:transpositions}. 

Of course, we can further restrict our attention to inequivalent
factorizations into transpositions.  For $m \leq 3$ we obtain simple
expressions for $\ikGS{m}{2}$ as special cases of $\ikGS{m}{k}$, see
Corollary~\ref{cor:transpositions}.  But we have also used a
specialization of Theorem~\ref{thm:digraphs} to obtain a compact form
for $\ikGS{4}{2}$, currently the only result available for $m=4$.

\begin{thm}
\label{thm:transpositions}
Let $\xis \in \Q[[x]]$ be the unique solution of $\xis = x(1-\xis)^{-2}$, namely
\begin{equation*}
	\xis(x) = \sum_{n \geq 1} \frac{1}{n} \binom{3n-2}{n-1} x^n.
\end{equation*}
Letting $\xis_i = \xis(x_i)$ and $e_{i} \equiv \esf{i}(\xis_1,\xis_2,\xis_3,\xis_4)$, we have
\begin{equation}
  \label{eq:4case_transposition}
   \begin{split}
    \ikGS{4}{2}
		&= 6(\Dopx+1) \sum_{i=1}^4 \frac{\xis_i}{1-3\xis_i} \prod_{j \neq i} 
			\frac{\xis_j}{(\xis_i-\xis_j)(1-\xis_i-\xis_j)}
			+
			\frac{12e_4(4-4e_1+3e_2)}{\prod_i (1-3\xis_i) \prod_{i < j} (1-\xis_i - \xis_j)}.
    \end{split}
\end{equation}
\end{thm}

The derivation of equation \eqref{eq:4case_transposition} is given in Section~\ref{sec:4case_transposition}.
As with $\icGS{m}$, we have not been able to deduce a universal form of $\ikGS{m}{2}$ valid for all $m$.

\subsection{Minimal Transitive Cycle Factorizations}
\label{sec:survey_cyclefacts}

More important than the enumerative content of Theorem~\ref{thm:Q} is its striking similarity with analogous results for ordinary cycle factorizations.

Let $\ocfacts{\a}{\b}$ be the number of minimal transitive cycle factorizations with signature $\b$ of any permutation with cycle type $\a$. In accordance with~\eqref{eq:icGS}, set
\begin{align}
\label{eq:ocGS}
	 \ocGS{m}(\bx,\bq) =
	 		\sum_{\a, \b} \ocfacts{\a}{\b} \frac{1}{|\b|!} \frac{x_1^{\a_1} \cdots x_m^{\a_m}}{\a_1 \cdots \a_m} \bq^{\b}
\end{align}
for $m \geq 1$. 

Closed form expressions for $\ocGS{m}$ are known only
when $m=1$ or $m=2$.  The case $m=1$, which corresponds to
factorizations of the full cycle, is well understood both
bijectively~\cite{Spr96} and algebraically
while the case $m = 2$ was treated in~\cite{Irv_cjm09} using a
graphical decomposition for cycle factorizations.  It transpires that
both series can be neatly expressed in terms of the unique solution $w
\in \Q[\bq][[x]]$ of
\begin{equation}
\label{eq:w_recur}
	w = xe^{\QQ(w)},
\end{equation}
where $Q$ is defined as in~\eqref{eq:Q}.  Letting $T(w)=1-wQ'(w)$,  the results of~\cite{Irv_cjm09} can be rewritten as
\begin{align}
		\Dopx \ocGS{1} &= w_1, 
				\label{eq:ordered_m1}
				\\
		\Dopx \ocGS{2}
			&= w_1 w_2 \pr{\frac{1}{T_1}-\frac{1}{T_2}}  
					\frac{\QQ_1-\QQ_2}{(w_1-w_2)^2}, 		
				\label{eq:ordered_m2}
	\end{align}
where $w_i = w(x_i,\bq)$, $\QQ_i = Q(w_i)$,  and $T_i=T(w_i)$ for $i=1,2$.  
We note in passing that the explicit formula~\eqref{eq:fullcyclefacts} for $\ocfacts{(n)}{\b}$ is readily derived  from~\eqref{eq:w_recur} and~\eqref{eq:ordered_m1} by Lagrange inversion.


  We now invite the reader to compare~\eqref{eq:Qm1} and~\eqref{eq:Qm2} with~\eqref{eq:ordered_m1} and~\eqref{eq:ordered_m2}. The connection is strong enough that we have used~\eqref{eq:Qm3} to model the following conjectural form for $\ocGS{3}$.

\begin{conjecture}
\label{conj:ordered3case}
With the same notation as in~\eqref{eq:w_recur}---\eqref{eq:ordered_m2}, we have
$$		
		\ocGS{3} = w_1 w_2 w_3 \sum_{i =1}^3 \frac{1}{T_i} \prod_{j \neq i} 
		\frac{\QQ_i-\QQ_j}{(w_i-w_j)^2}. 
$$
\end{conjecture}

We have tested this conjecture against sufficient data to be confident in its truth.
In fact, we are certain it can be proved by generalizing the graphical constructions in
~\cite{Irv_cjm09}, but we feel that the insight gained by such a proof is unlikely to be worth working through the technical details to obtain it. Ultimately, new methods will be required to shed further light on the nature of this connection between inequivalent and ordered factorizations.

Let $\okGS{m}{k}$ be defined analogously to $\ikGS{m}{k}$. Then it is straightforward to verify that Hurwitz's formula~\eqref{eq:hurwitz} is equivalent to 
\begin{equation}
\label{eq:hurwitz2}
 	\okGS{m}{2}(\bx) = \Dopx^{m-3} \prod_{i=1}^m \frac{w_i}{1-w_i}, \qquad m \geq 1.
\end{equation}
For arbitrary $k$, closed form expressions for $\okGS{m}{k}$ are known only when $m \leq 3$.  These were found by Goulden and Jackson~\cite{GouJac_ejc00}, and Conjecture~\ref{conj:ordered3case} does indeed specialize  to their results. (Verifying this fact in the case $m=3$ is best done using Lemmas~\ref{lem:3case_simplification} and~\ref{lem:det} of the Appendix.)

Note that the passage from~\eqref{eq:Qm3} to
Conjecture~\ref{conj:ordered3case} is essentially effected by setting
$\QS \equiv 1$ and replacing $\xis$ with $w$ and $\QD(\xis)$ with
$T(w)$. This same correspondence appears upon
implicit differentiation of the defining equations~\eqref{eq:Q_xis}
and~\eqref{eq:w_recur}, which yields
\begin{equation}
  \label{eq:xdx}
		x \pd{w}{x} = \frac{w}{T}.
		\qquad\text{and}\qquad
		x \pd{\xis}{x} = \frac{\xis \QS}{\QD}.
\end{equation}
These relations indicate that $\QS$ and $\QD$ are not as arbitrary as they may at first appear.   

\begin{remark}
  In comparing~\eqref{eq:Qm3} to Conjecture~\ref{conj:ordered3case},
  the reader will observe an extra factor of 2 that is not explained
  by the substitutions described above.  The second author and
  G. Chapuy are currently working on a unified framework for
  inequivalent and ordered cycle factorizations which would explain this
  factor and many other aspects of the connection observed above.
  These results will be reported elsewhere.
\end{remark}

\section{Relationships with Other Classes of Factorizations}
\label{sec:connections}

\newcommand{\mono}[1][n]{\mathcal{M}_{#1}}
\newcommand{\ineq}[1][n]{\mathcal{I}_{#1}}
\newcommand{\proper}[1][n]{\mathcal{P}_{#1}}
\newcommand{\cycfacts}[1][n]{\mathcal{C}_{#1}^*}
\newcommand{\cycles}[1][n]{\mathcal{C}_{n}}
\newcommand{\words}{\mathcal{W}_n}
\newcommand{\eval}[1]{\Pi(#1)}
\newcommand{\mx}[1]{\max(#1)}
\newcommand{\mxf}[1][f]{\hat{#1}}
\newcommand{\tr}[1]{\mathrm{tr}(#1)}
\newcommand{\discyc}{\overline{\mathcal{I}}_n}
\newcommand{\Proj}{P_{\mathcal{I}}}

Let $\ineq$, $\proper$ and $\mono$, respectively, be the sets of all inequivalent, proper, and monotone factorizations in $\Sym{n}$.  (See Section~\ref{sec:survey_connections} for the relevant definitions.)  In this section we shall develop connections between these sets.  In doing so, it will be convenient to view them as monoids under concatenation.  Thus we consider each of them to contain the \emph{empty factorization}, denoted by 1, which is a factorization of the identity permutation $1 \in \Sym{n}$.  

Let $f=(\s_1,\ldots,\s_r)$ be any factorization. Let $\target{f}$ denote the target permutation $\s_1\s_2\cdots\s_r$, and extend this definition so that $\Pi$ 
acts linearly on formal sums of factorizations.   Note that $\target{f}$ is well defined for inequivalent factorizations $f$.      Let $\b(f)$ be the
signature of $f$, with $\b(1)=\mathbf{0}$. Similarly define the signature
$\b(\s)$ or any permutation $\s$.  Clearly 
$\b(fg)=\b(f)+\b(g)$, where $fg$ is the concatenation of
factorizations $f$ and $g$.

\subsection{Inequivalent and Proper Factorizations}

Let $\cycles$ be the set of all nontrivial cycles in $\Sym{n}$, and let $\cycfacts$  be the set of all words on this alphabet.  Every permutation in $\Sym{n}$ can be viewed as an element of $\cycfacts$ by listing its nontrivial disjoint cycles in increasing order of least element, with  the identity
permutation corresponding with the empty word.
Every proper factorization $f \in \proper$ is also then associated with an element of $\cycfacts$ by concatenating the words  of its factors. 
In particular, this restricts to a natural one-one correspondence between $\cycfacts$ and the set of all cycle factorizations in $\Sym{n}$ (including the empty factorization). 

\newcommand{\proj}[1]{\left[ #1 \right]}
On $\cycfacts$ we have the equivalence relation $\sim$ induced by
allowing commutations of adjacent disjoint cycles, and the quotient $\cycfacts/\sim$ is naturally identified with $\ineq$.  Let $f \mapsto \proj{f}$ be the canonical projection of $\cycfacts$ onto $\ineq$, extended linearly to all of $\Q\cycfacts$.
For example,
\begin{equation*}
	\big[2(1\,2)(3\,4) + (3\,4)(1\,2)\big] = 3\big[(1\,2)(3\,4)\big].
\end{equation*}
In the following proof, elements of $\Sym{n}$, $\proper$ and $\ineq$ should be viewed as words on $\cycfacts$ (and their projections).

\begin{thm}
\label{thm:propercorrespondence}
In  $\Q\Sym{n}[[\bq]]$ we have
\begin{equation}
\label{eq:qidentity}
	\Big( \sum_{\s \in \Sym{n}} \bq^{\b(\s)} \s \Big)^{-1}
	=
	\sum_{f \in \proper} \target{f} (-1)^{\len{f}} \bq^{\b(f)}
	=
	\sum_{f \in \ineq} \target{f} (-1)^{\len{f}} \bq^{\b(f)}.
\end{equation}
Moreover, the rightmost identity continues to hold if $\proper$ and $\ineq$ are restricted to contain only transitive factorizations of any fixed genus.
\end{thm}

\begin{proof}
Let $\discyc$ be the subset of $\ineq$ consisting
of all nonempty words on $\cycles$ whose letters commute
pairwise.  Then the Cartier-Foata theorem~\cite{cartierfoata} yields the following identity in $\Q\ineq[n][[\bq]]$:
\begin{equation*}
  \Big( 1 + \sum_{f \in \discyc} (-1)^{\len{f}} \bq^{\b(f)} f\Big)^{-1}
  =
  \sum_{f \in \ineq} \bq^{\b(f)} f.
\end{equation*}
Note that $(-1)^{\len{f}} \bq^{\b(f)}=(-\bq)^{\b(f)}$ for  $f \in \discyc$. Moreover, each $f \in \discyc$ 
corresponds with a distinct $\s \in \Sym{n}$ of the same signature.   Thus 
\begin{equation}
  \label{eq:cartierfoata}
	\Bigg[\Big( \sum_{\s \in \Sym{n}} (-\bq)^{\b(\s)} \s \Big)^{-1}\Bigg]
	=
	\sum_{f \in \ineq}  \bq^{\b(f)} f.
\end{equation}
On the other hand, expansion in $\Q\cycfacts[n][[\bq]]$  gives
\begin{equation}
\label{eq:fq}
	\Big( \sum_{\s \in \Sym{n}} \bq^{\b(\s)} \s \Big)^{-1}
= \sum_{k \geq 0} (-1)^k \Big( \sum_{\s \neq 1} \bq^{\b(\s)}\s\Big)^{\!\! k}
  = \sum_{f \in \proper} (-1)^{\len{f}} \bq^{\b(f)} f.
\end{equation}
Together, \eqref{eq:cartierfoata} and~\eqref{eq:fq} yield the $\Q\ineq[n][[\bq]]$ identity
\begin{align}
	\Bigg[	\Big( \sum_{\s \in \Sym{n}} \bq^{\b(\s)} \s \Big)^{-1} \Bigg]
	=
  \Bigg[ \sum_{f \in \proper} (-1)^{\len{f}} \bq^{\b(f)} f \Bigg]
  &=
  \sum_{f \in \ineq} (-1)^{\len{f}} \bq^{\b(f)} f.
  \label{eq:fq3}
\end{align}
Now the image of~\eqref{eq:fq3} under $\Pi$ is precisely~\eqref{eq:qidentity}, since clearly $\target{[f]} = [\target{f}]$. Also notice
that the rightmost identity in~\eqref{eq:fq3} continues to hold
when we restrict $\proper$ and $\ineq$ to include only transitive
factorizations, since a proper factorization is transitive if and only
if its induced cycle factorization is transitive. The same is therefore true of~\eqref{eq:qidentity}.  We can further
restrict to any particular genus simply by discarding all terms with
inappropriate signatures.
\end{proof}

\begin{example}
Consider the word $(1\,5)(2\,4)(3\,5) \in \cycfacts[5]$. The corresponding inequivalent factorization (\emph{i.e.} equivalence class) in $\ineq[5]$ is
$$
	f = \{ (1\,5)\cdot(2\,4)\cdot(3\,5),\quad
			 (2\,4) \cdot (1\,5) \cdot (3\,5),\quad
			 (1\,5) \cdot (3\,5) \cdot (2\,4)\},
$$  
There are precisely five factorizations in $\proper[5]$ which project to $f$, namely
  \begin{xalignat*}{2}
    &(1\,5)\cdot(2\,4)\cdot(3\,5) &     &(1\,5)(2\,4)\cdot(3\,5) \\
    &(2\,4)\cdot(1\,5)\cdot(3\,5) &     &(1\,5)\cdot(2\,4)(3\,5) \\
    &(1\,5)\cdot(3\,5)\cdot(2\,4) & &
  \end{xalignat*}
Each factorization in the first column contributes $(-1)^3 q_2^3\, f$ to the middle sum of~\eqref{eq:fq3}, while those in the second contribute $(-1)^2 q_2^3\,f$. The total contribution is therefore $-q_2^3\,f$, which matches the contribution $(-1)^3 q_2^3\,f$ that $f$ makes to the right-hand sum of~\eqref{eq:fq3}. \qed
\end{example}

We note that it is straightforward to ``combinatorialize'' the proof of Theorem~\ref{thm:propercorrespondence}, but in doing so one is effectively recreating the proof of the Cartier-Foata identity.

\subsection{Inequivalent and Monotone Factorizations}

We now establish a connection between $\ineq$ and $\mono$.  We give two proofs of this result, one combinatorial and one algebraic, as we believe both provide insight into the underlying structure. Our combinatorial proof is an adaptation of a similar proof for semiclassical diagrams appearing in \cite{BerKui_pre12,BerKui_jmp13a}, where it is described using modifications of maps.  

\begin{thm}
\label{thm:primcorrespondence}
In $\Q\Sym{n}[[u]]$ we have
\begin{equation}
\label{eq:primidentity}
	\sum_{f \in \mono} \target{f} u^{\len{f}}
	=
	\sum_{f \in \ineq} \target{f} (-1)^{\len{f}} (-u)^{\depth{f}}.
\end{equation}
This identity continues to hold if $\mono$ and $\ineq$ are replaced with the subsets thereof
consisting of transitive factorizations of any fixed genus.
\end{thm}

\begin{proof}[Combinatorial Proof:]
Define the \emph{trace} of a cycle factorization $f = (\s_1,\cdots,\s_r)$ by 
\begin{equation*}
  \tr{f} := (\mx{\s_1},\ldots,\mx{\s_r}) \in [n]^r,  
\end{equation*}
where $\mx{\s_i}$ is the largest element of cycle $\s_i$.
Then every inequivalent factorization (\emph{i.e.} equivalence class) has a unique representative $f$ whose trace is minimal in the usual lexicographic order.  We identify $\ineq$ with the set of these canonical forms and define $\map{I}{\ineq}{\ineq}$ as follows. 

If  $f$ is a monotone factorization, set $I(f)=f$.  Now suppose $f=(\s_1,\ldots,\s_r)$ is not monotone. Let $\s_i$ be the leftmost factor such that $\mx{\s_i} > \mx{\s_{i+1}}$, if it exists, and otherwise let $\s_i$ be the leftmost factor that is \emph{not} a transposition.  Let $m = \mx{\s_i}$.  There are two cases to consider:
\begin{enumerate}
\item If $\s_i = (a\, m)$, then the minimality of $\tr{f}$ and the
  condition $\mx{\s_i} > \mx{\s_{i+1}}$ imply that $\s_i$ and
  $\s_{i+1}$ do not commute and thus force
  $\s_{i+1}=(a\,b_1\,\cdots\,b_k)$ for some $b_1,\ldots,b_k < m$. We
  modify $f$ by multiplying the factors $\s_i$ and $\s_{i+1}$
  together, i.e. define
$$
	I(f)=\s_1\cdots\s_{i-1}(a\,b_1\cdots b_k\,m)\s_{i+2}\cdots\s_r.
$$

\item	 If $\s_i = (a\,b_1\cdots b_k\,m)$, then we define
$$
I(f)=\s_1\cdots\s_{i-1}(a \,m)(a\,b_1\cdots b_k)\s_{i+1}\cdots\s_r.
$$
\end{enumerate}
It is easy to check that $\tr{I(f)}$ is minimal and $I(I(f))=f$. Thus $I$ is an involution on $\ineq$, and its fixed points are the monotone factorizations.  Clearly $\target{I(f)}=\target{f}$ and $\depth{I(f)}=\depth{f}$. When $I(f) \neq f$ we have $\len{I(f)}=\len{f}\pm 1$, so that
  \begin{equation*}
    \target{I(f)} (-1)^{\len{I(f)}} (-u)^{\depth{I(f)}} = -\target{f}
    (-1)^{\len{f}} (-u)^{\depth{f}}.
  \end{equation*}
All factorizations that are not fixed points of $I$ therefore cancel each other in the right-hand sum of \eqref{eq:primidentity}, which proves its equality to the left-hand sum.  Moreover,  transitivity and genus are clearly preserved by $I$, so~\eqref{eq:primidentity} still holds when the sums are restricted by these conditions.  \end{proof}

\begin{proof}[Algebraic Proof]
Let $J_2,\ldots,J_n$ be the Jucys-Murphy elements in $\Q\Sym{n}$, defined by $J_k = \sum_{i=1}^{k-1} (i\,k)$.  Then  in $\Q\Sym{n}[[u]]$ we have
\begin{align*}
	\sum_{f \in \mono} \target{f} u^{\len{f}}
	= \sum_{i_j \geq 0} u^{i_2+\cdots+i_n} J_2^{i_2} J_3^{i_3} \cdots J_n^{i_n}  
	= \prod_{i=2}^n (1-uJ_i)^{-1}.
\end{align*}
Since the $J_i$ commute it follows that
\begin{align*}
	\sum_{f \in \mono} \target{f} u^{\len{f}}
	= \Big( \prod_{i=2}^n (1-uJ_i) \Big)^{-1} 
	= \Big(1 + \sum_{k =1}^{n-1} (-u)^k e_k(J_2,\ldots,J_n)\Big)^{-1}.
\end{align*}
But it is well known that $e_{k}(J_2,\ldots,J_n)$ 
evaluates to the sum of all permutations
in $\Sym{n}$ composed of $n-k$ cycles; 
in other words, all $\s \in \Sym{n}$ with $\depth{\s}=k$. Thus
\begin{equation*}
	\sum_{f \in \mono} \target{f} u^{\len{f}} = \Big(\sum_{\s
          \in \Sym{n}} (-u)^{\depth{\s}} \s \Big)^{-1}.
\end{equation*}
Equation~\eqref{eq:primidentity} now follows from~\eqref{eq:qidentity} upon setting $q_k = (-u)^{k-1}$.

That~\eqref{eq:primidentity} continues to hold when restricted to transitive subsets of factorizations is a consequence of two elementary observations.  First, rearranging the factors of a factorization preserves transitivity. 
Second, if $f=(\s_1,\ldots,\s_r)$ is any factorization in $\Sym{n}$ and $f_i$ is a factorization of $\s_i$ into the minimal number $\depth{\s_i}=n-\len{\s}$ of transpositions, then $f$ is transitive if and only if $f_1 \cdots f_r$ is transitive. We can further restrict~\eqref{eq:primidentity} to any genus by selecting appropriate powers of $u$.
\end{proof}

\subsection{Proof of Theorem~\ref{thm:connections}}  Fix any permutation $\p$ of cycle type $\a$.  By~\eqref{eq:rank}, every genus $g$ factorization of $\p$ has depth $d =|\a|+\len{\a}-2+2g$.  Restrict Theorem~\ref{thm:primcorrespondence} to factorizations of genus $g$ and extract the coefficient of $\p u^d$ to obtain
\begin{equation}
\label{eq:monoineq}
	\mgfacts{\a} = (-1)^{d} \sum_f
	 (-1)^{\len{f}},
\end{equation}
where the sum extends over all genus $g$ inequivalent factorizations $f$ of $\p$.  
This is the leftmost equality of Theorem~\ref{thm:connections}. Setting $\bq = \mathbf{1}$ in Theorem~\ref{thm:propercorrespondence} provides the middle equality.

\subsection{Generating Series for Monotone Factorizations}
We conclude this section by restating the relationship between monotone and inequivalent factorizations in terms of familiar generating series. For simplicity we will restrict our attention to the minimal transitive case, but the obvious analogues hold in any genus.  

Let $\p$ be a fixed permutation of type $\a=(\a_1,\ldots,\a_m)$.  Then~\eqref{eq:monoineq} identifies $\mgfacts[0]{\a}$ as the sum of
$
	(-1)^{\len{f}+\depth{f}}
$
over all inequivalent minimal transitive factorizations $f$ of $\p$.   But this is precisely the coefficient of $\prod x_i^{\a_i}/{\a_i}$ in the series $\icGS{m}(\mathbf{x},\mathbf{q})$ evaluated at
$q_k = (-1)^k$, since $\b(f)=(\b_2,\b_3,\ldots)$ implies $ \len{f}+\depth{f}=\sum_{k \geq 2} k\b_k$.  

Let us define
\begin{equation*}
	 \pGS{m}(\bx) = 
	 		\sum_{\a} \mgfacts[0]{\a} \frac{x_1^{\a_1} \cdots x_n^{\a_m}}{\a_1 \cdots \a_m}
\end{equation*}
in analogy with~\eqref{eq:icGS} and~\eqref{eq:ocGS}. Then we have the following algebraic connection between monotone and inequivalent factorizations.
\begin{cor}
\label{cor:monotone}
For all $m \geq 1$,
\begin{equation*}
\pGS{m}(\bx)
= \icGS{m}(\bx,\bq) \Big|_{q_k = (-1)^k}.
\end{equation*}
\end{cor}

This allows us to apply Theorem~\ref{thm:Q} to compute $\pGS{m}$ for $m \leq 3$.  If we set $q_k=(-1)^k$ in~\eqref{eq:Q} then~\eqref{eq:Q_xis} becomes $\xis = x(1+\xis)^{2}$, thereby identifying $\xis$  as the generating series of the Catalan numbers:
$$
	\xis(x) = \sum_{n \geq 1} \frac{1}{n+1}\binom{2n}{n} x^n
			= \frac{1-\sqrt{1-4x}}{2x}-1.
$$
Corollary~\ref{thm:goulden} is readily seen to be equivalent to
\begin{equation}
\label{eq:pgs_general}
	\pGS{m}(\bx) = (2\Dopx+1)^{(m-3)} \prod_{i=1}^m \frac{2\xis_i}{1-\xis_i},
\end{equation}
and indeed setting $q_k=(-1)^k$ throughout Theorem~\ref{thm:Q} gives this result for $m \leq 3$.  

We have been unable to use Corollary~\ref{cor:monotone}
and~\eqref{eq:pgs_general} to deduce anything substantive about the
structure of $\icGS{m}$. However, given the strength of the
connections that we have established, comparison of
\eqref{eq:hurwitz2} and~\eqref{eq:pgs_general} leads us to conjecture
that $\ikGS{m}{2}$ may, too, be expressed ``compactly'' in terms of an
$(m-3)$-times iterated differential operator.  This conjecture has
motivated our search for the expression for $\ikGS{4}{2}$ which will
be presented in Theorem~\ref{thm:transpositions} below.

\section{Graphical Models for Inequivalent Factorizations}
\label{sec:models}

In this section we establish Theorem~\ref{thm:digraphs} and then restate the result in a manner more convenient for our derivation of Theorem~\ref{thm:Q}.

\subsection{Shuttle Diagrams and Alternating Maps}
We begin with a nice visualization of a product $(i_1\,j_1)(i_2\,j_2)\cdots(i_r j_r)$ of transpositions in $\Sym{n}$, originally
suggested in \cite{Gar59} (see also \cite{Bog08}).
First, draw $n$ horizontal arrows, directed from right to left and labelled from 1 to $n$. Then connect these arrows in pairs using $r$ vertical lines (``shuttles''), with one shuttle between arrows $i_k$ and $j_k$ for each transposition $(i_k\,j_k)$, and such that the right-to-left order of the shuttles matches that of the factors in the product. 
See Fig.~\ref{fig:shuttlediagram}a for an illustration.
\begin{figure}[t]
  \centering
  \includegraphics[height=3.5cm]{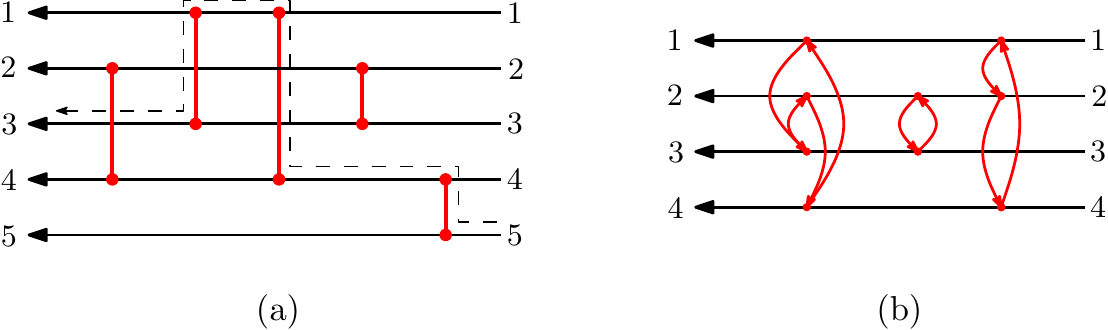}
  \caption{(a) Visualizing the product $(2\,4)(1\,3)(1\,4)(2\,3)(4\,5)$ via a ``shuttle
    diagram''.  The image of 5 under the product
    is indicated by the dashed line.
	(b) General cycle factors are represented by ``multi-shuttles'', illustrated here
	for the product $(1\,3\,2\,4) \cdot (2\,3) \cdot (1\,2\,4)$. }
  \label{fig:shuttlediagram}
\end{figure}

Observe that the image of symbol $i$ under the product is found by beginning at the tail of arrow $i$ and tracing to the left, following shuttles whenever encountered, until the head of an arrow is reached.  The label of this terminal arrow is the image of $i$. 

This construction is easily extended to allow for products of cycles of any length:  A $k$-cycle factor $(i_1\, i_2\, \cdots\, i_k)$ is represented by a ``multi-shuttle'' joining arrows $i_1, i_2,\ldots,i_k$ in cyclic order, as demonstrated in Fig.~\ref{fig:shuttlediagram}b.   Thus we have a simple correspondence between cycle factorizations and ``shuttle diagrams''.

Let $f$ be a cycle factorization.  From its shuttle diagram,
construct a labelled digraph $\fgraph{f}$ as follows:  First, place a
vertex at the tail and head of each arrow and assign each of these
vertices the same label as the arrow. Note that the shuttles subdivide
the arrows into segments. Assign each such segment the label and
direction inherited from its arrow.  Finally, contract each shuttle
into a single vertex to obtain $\fgraph{f}$. Fig.~\ref{fig:shrinking}
illustrates this process.

\begin{figure}[t]
  \centering
  \includegraphics[width=0.8\textwidth]{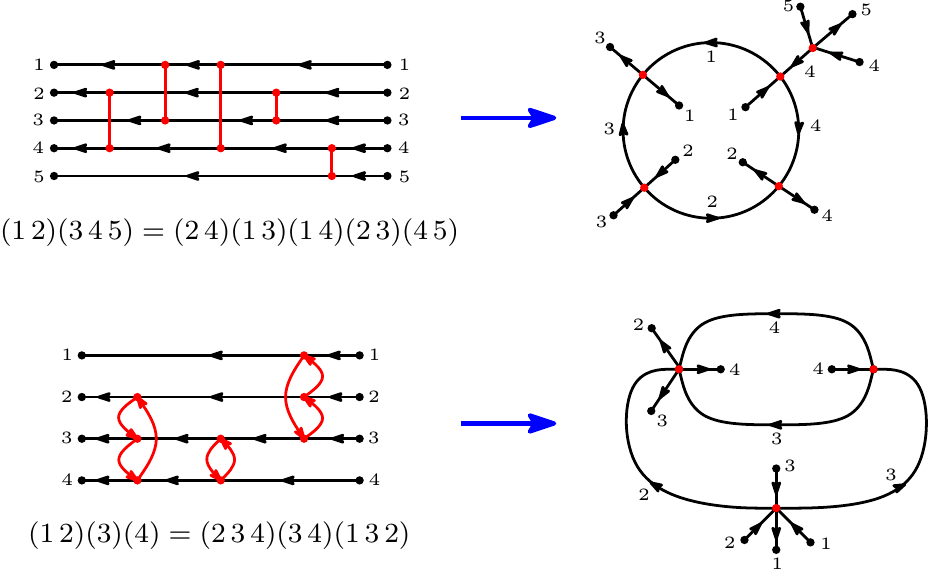}
  \caption{Creating a labelled digraph from a shuttle diagram by
    collapsing each shuttle to a single vertex.  Labels of the edges
    incident to leaves coincide with the leaf labes and are omitted.}
  \label{fig:shrinking}
\end{figure}

We now associate with $f$ an alternating map $\fmap{f}$ by specifying an embedding of $\fgraph{f}$ in an orientable surface. Recall that such an embedding is fully determined by the cyclic order of edges around internal (non-leaf) vertices~\cite{mohar-thomassen}. Each internal vertex of $\fgraph{f}$ arises as a collapsed shuttle. In particular, the shuttle corresponding to  factor $(i_1\,i_2\,\cdots\,i_k)$ yields a vertex $v$ having $k$ in-directed edges labelled $i_1,\ldots,i_k$ and $k$ out-directed edges labelled the same.  The map $\fmap{f}$ is obtained by insisting that these edges are arranged around $v$ so that, when listed in counter-clockwise order, their labels are $(i_1,i_1,i_2,i_2,\ldots,i_k,i_k)$ with alternating directions $(\mathit{out}, \mathit{in}, \mathit{out},\ldots)$.  Fig.~\ref{fig:localrules} illustrates this local embedding rule. Notice that the planar embeddings in Fig.~\ref{fig:shrinking} obey this rule, making them the alternating maps associated with the given factorizations.
\begin{figure}[t]
  \centering
  \includegraphics[width=0.9\textwidth]{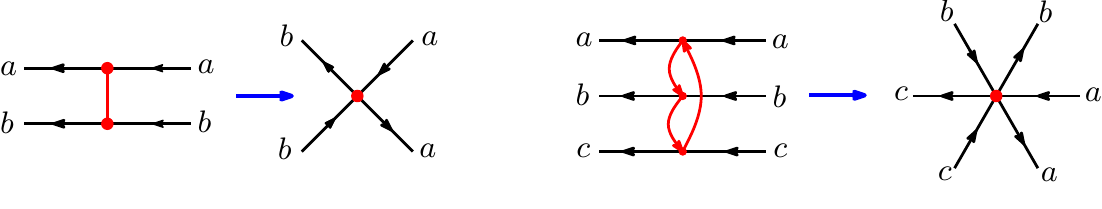}
  \caption{Local embedding rules for vertices arising from the transposition $(a\,b)$ and 3-cycles $(a\,b\,c)$. 
  These rules have been observed in the planar embeddings of Fig.~\ref{fig:shrinking}.
}
  \label{fig:localrules}
\end{figure}

Clearly the factors of $f$ can be recovered from $\fmap{f}$.  However, we cannot generally determine the \emph{order} of the factors, so the transformation $f \mapsto \fmap{f}$ is not fully reversible.  Fortunately, it fails to be injective in a very convenient manner.
 
\begin{lem}
\label{lem:equaldigraphs}
$\fmap{f}=\fmap{f'}$ if and only if $f \sim f'$
\end{lem}

\begin{proof}
The equivalence class of a factorization $f=(\s_1,\ldots,\s_r)$ is uniquely determined by the multiset $[\s_1,\ldots,\s_r]$ of factors 
and the relative orderings of the factors $[\s_j \,:\, \s_j(i) \neq i]$ that move symbol $i$, for all $i$.   But the factors moving $i$  correspond with the shuttles incident with arrow $i$ in the shuttle diagram of $f$, and the relative ordering of these shuttles is encoded by $\fmap{f}$. Indeed, their right-to-left order is that in which the corresponding vertices are encountered on the unique directed path in $\fmap{f}$ that connects the two leaves labelled $i$ by edges of the same label.
\end{proof}

We require one further preliminary result before proving Theorem~\ref{thm:digraphs}.

\begin{lem}
\label{lem:sources}
Every face of an alternating map contains an equal number of sources and sinks. If an alternating map is acyclic (\emph{i.e.} has no directed cycles) then each of its faces contains both a source and a sink.
\end{lem}

\begin{proof}
Let $F$ be a face of an alternating map. Recall that the \emph{boundary walk} $W$ of $F$ is the closed walk traversing the boundary of $F$ and keeping $F$ on the right, relative to the direction of traversal.

The edges encountered along $W$ are directed either forward (in the direction of $W$) or backward, and the alternating condition implies that a change in direction occurs at a vertex $w$ on $W$ if and only if $w$ is a source or a sink. In particular, the direction changes from forward to backward at a sink, and from backward to forward at a source. As a result, the segment of $W$ beginning at any source $r$ will be a forward-directed path from $r$ to a sink, and the segment beginning at a sink $s$  will be  a backward-directed path from $s$ to a source, \emph{etc.}   Therefore $F$ contains the same number of sources as sinks.  Moreover, if $F$ contains neither a source nor a sink then $W$ must be a directed closed walk, and therefore contains a directed cycle. 
\end{proof}

\subsection{Proof of Theorem~\ref{thm:digraphs}}
Let $\p \in \Sym{n}$ be of cycle type $\a=(\a_1,\ldots,\a_m)$, and let $f$ be a genus $g$ cycle factorization of $\p$ of length $r$ and signature $\b$. 

Clearly $\fmap{f}$ is alternating. It is also acyclic, since (by construction) there exists a directed path from internal vertex $u$ to internal vertex $v$ if and only if  the factor of $f$ corresponding to $u$ appears to the right of the factor of $f$ corresponding to $v$. This precludes the existence of both a directed path from $u$ to $v$ \emph{and} one from $v$ to $u$.

Lemma~\ref{lem:sources} guarantees every face of $\fmap{f}$ contains at least
one source.  Choose any source $r$ in a face $F$ and suppose it
has label $i$. As in the proof of Lemma~\ref{lem:sources}, the boundary walk of $F$
contains a directed path $P$ from $r$ to a sink $s$, followed by a
reverse-directed path $P'$ from $s$ to a source $r'$.  In fact, $P$ is simply the path
traversed when using the shuttle diagram of $f$ to determine $\p(i)$,
and $P'$ is a backwards traversal of arrow $\p(i)$.  Hence $s$ and
$r'$ both have label $\p(i)$. Repeating this argument, we conclude
that the sources (and sinks) of face $F$ have labels $i, \p(i),
\p^2(i), \ldots$ when listed in the direction of the boundary walk.
In particular, the faces of $\fmap{f}$ correspond with the cycles of
$\p$.  Therefore $\fmap{f}$ has $\len{\p}$ faces.

Observe that $\fmap{f}$ has $2n+\sum_k \b_k$ vertices, with $n$
sources, $n$ sinks, and $\b_k$ vertices of degree $2k$, $k\geq 2$.
Thus it has $\frac{1}{2}(\sum_k 2k\b_k + 2n)= \depth{f}+n+\sum_k \b_k$
edges. The Euler-Poincare formula therefore shows $\fmap{f}$ to be of
genus $\frac{1}{2}(2-n+\depth{f}-\len{\p})$, which by~\eqref{eq:rank}
evaluates to the genus $g$ of the factorization.

We now redecorate $\fmap{f}$ to obtain a map satisfying the conditions of Theorem~\ref{thm:digraphs}. First delete all edge labels from $\fmap{f}$, observing that the embedding
rules make them recoverable from the labels of the sources/sinks.  
Let $C_1,\ldots,C_m$ be a canonically ordered list of the cycles of
$\p$, with $C_i$ of length $\a_i$, and let $c_i$ be the least element
of $C_i$.  For each $i$, find the
unique source labelled $c_i$, distinguish its position, and assign
label $i$ to the ambient face.  We have seen that the label of any source determines
the labels of all sources/sinks in the same face, so we can now delete the labels of all sources and sinks with no loss of information.
This results in a map $\fmap{f}'$ satisfying the conditions of Theorem~\ref{thm:digraphs}.  Since the passage from $\fmap{f}$ to $\fmap{f}'$ is reversible, Lemma~\ref{lem:equaldigraphs} shows that the equivalence class of
$f$ can be uniquely recovered from $\fmap{f}'$.  This completes the proof.

\subsection{An Undirected Analog of Theorem~\ref{thm:digraphs}}

When we turn to counting factorizations in the next section, it will be more
convenient to deal with undirected maps than acyclic alternating maps. 
Clearly the directions of all edges in an alternating map are determined by the direction of any one edge. The key to stripping edge directions lies in finding an appropriate analogue of the
acyclic condition.

Let $W$ be a walk in an orientable map. A \emph{corner} of $W$ is an ordered pair of consecutive edges of $W$. We say that a corner $(e,e')$ is \emph{odd} if a path at vanishingly small distance to the right of $W$ crosses an odd  number of  edges as it shadows $W$ along $e$ and $e'$.  A corner is \emph{even} if it is not odd.  (See Fig.~\ref{fig:oddcorners}.)  
\begin{figure}[t]
  \centering
  \includegraphics[scale=0.8]{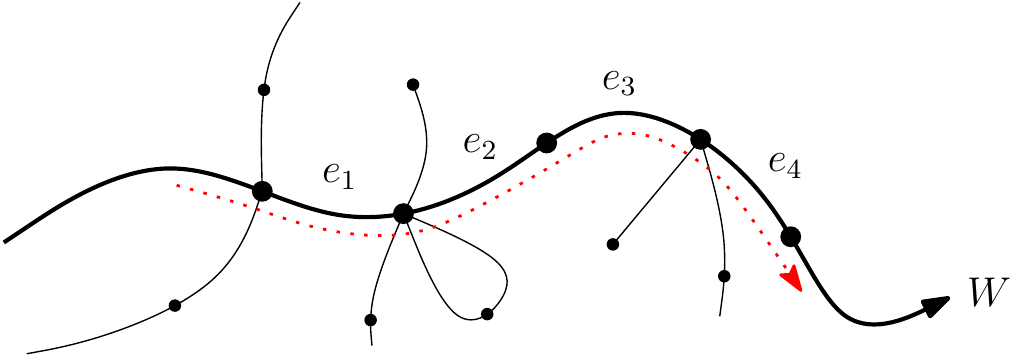}
  \caption{Corner $(e_1,e_2)$ of walk $W$ is odd, whereas $(e_2,e_3)$ and $(e_3,e_4)$ are even.}
  \label{fig:oddcorners}
\end{figure}
With this definition in hand, we have the following restatement of Theorem~\ref{thm:digraphs}.

\begin{thm}
  \label{thm:even_degree}
  Let $\a=(\a_1,\ldots,\a_m)$ be an $m$-part composition of $n$ and let $\p \in \Sym{n}$ be of cycle type $\a$.
  Minimal inequivalent cycle factorizations of $\p$
  with signature $\b=(\b_2,\b_3,\ldots)$
  are in $\a_2\cdots\a_m$-to-one correspondence with planar maps with
  $m$ labelled faces in which 
  \begin{enumerate}
  \item[(a)] every vertex is a leaf or has even degree $\geq 4$,
  \item[(b)] face $i$ contains exactly $2\a_i$ leaves
  \item[(c)] one leaf in face $1$ is distinguished,
  \item[(d)] there are a total of $\b_k$ vertices of degree $2k$, for all $k \geq 2$, and
  \item[(e)] every cycle has a positive even number of odd corners.
  \end{enumerate}
  Genus $g$ inequivalent cycle factorizations are in
  correspondence with genus $g$ maps in the
  same fashion as above.
\end{thm}

\begin{proof}
  We describe a simple correspondence between the directed maps of
  Theorem~\ref{thm:digraphs} and the undirected ones described above.
  Given such an undirected map $\mathcal{M}$, the distinguished leaf
  in face $1$ is declared to be a source and its edge is directed accordingly,
  away from the leaf.  This choice of direction is then propagated around
  each vertex of $\mathcal{M}$ in alternating fashion to obtain an alternating map
  $\mathcal{M}'$.  The fact that there are an even number of odd
  corners along every cycle ensures that no inconsistencies arise
  in doing so.  Moreover, since the edges along a cycle in an
  alternating map must change direction at an odd corner, the fact
  that every cycle of $\mathcal{M}$ has at least one odd corner
  implies that $\mathcal{M}'$ is acyclic.  The sources in each
  face of $\mathcal{M}'$ are readily identified, and one source in each of faces
  $2,3,\ldots,m$ is chosen to be distinguished in $\a_2\cdots\a_m$ ways.
  The remaining properties are in direct correspondence.
\end{proof}


  
A remark on condition (c) of Theorem~\ref{thm:even_degree} is warranted. Since face-labelled maps with $m \geq 3$ faces have no nontrivial symmetries, a vertex in face 1 can be distinguished arbitrarily, resulting in a $2\a_1\a_2\cdots\a_m$-to-one correspondence between factorizations and maps satisfying all conditions of the theorem \emph{except} (c).  This is not true in cases $m=1$ and $m=2$, since planar maps with one or two labelled faces can have nontrivial automorphisms (such as rotational symmetry).


%


\begin{figure}[t]
  \centering
  \includegraphics[scale=0.85]{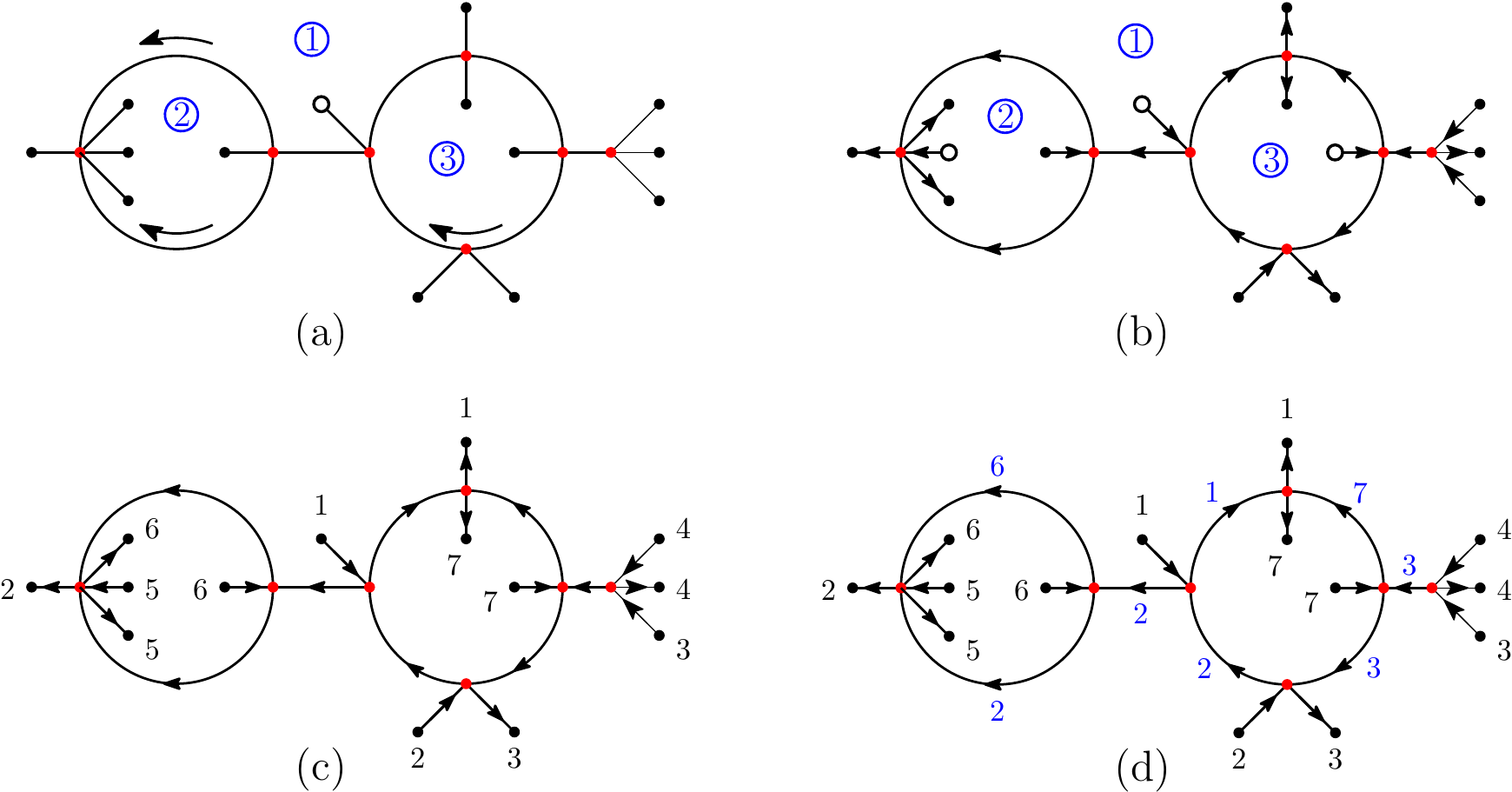}
  \caption{Reconstructing a factorization from a map: Recovering edge
    directions and labels.}
  \label{fig:reconstruction1}
\end{figure}

\begin{example}
  To illustrate Theorem~\ref{thm:even_degree} let us reconstruct
  a factorization from the map $\mathcal{M}$ shown in  Fig.~\ref{fig:reconstruction1}a.
	(For convenience we have indicated the orientations of boundary walks.)
  Note that $\mathcal{M}$
  satisfies the conditions of the theorem, with parameters $(\a_1,\a_2,\a_3)=(4,2,1)$
  and $(\b_2,\b_3,\ldots)=(6, 1,0,\ldots)$. It should
	therefore correspond   with $\a_2 \cdot \a_1 = 2$ 
  inequivalent factorizations of the permutation $(1\,2\,3\,4)(5\,6)(7)$. 

Begin by directing all edges of $\mathcal{M}$
so as to make the distinguished leaf a source and the entire map alternating. This can be done uniquely.  Then distinguish one source in each of faces 2 and 3. This can be done in $\a_2\cdot\a_3=2$ ways, one of which is shown in Fig.~\ref{fig:reconstruction1}b.
 
  Now label the leaves of face 1.  Begin at the distinguished source,
  giving it label 1, and then assign labels $2,2,3,3,4,4,1$ to the remaining leaves in the order in which they are encountered along the boundary walk.
Repeat this labelling procedure in faces 2 and 3 using label sets $\{5,6\}$ and $\{7\}$, 
  respectively, and then remove face labels. See Fig.~\ref{fig:reconstruction1}c.

Next, assign label $j$ to each edge of the boundary walk from sink $j$ to source $j$, for all $j$. (Note that this is always a reverse-directed path.)  See Fig.~\ref{fig:reconstruction1}d, where we suppress labels of edges incident with leaves, as such edges share the label of their leaf.

Finally, label each internal vertex with the cycle obtained by listing the labels of its incident out-directed edges as they appear in counter-clockwise order around the vertex.
Remove all leaves and edge labels to obtain an acyclic digraph whose nodes are labelled with cycles. (Fig.~\ref{fig:reconstruction2}.) This digraph induces a partial order on its vertices, with $u < v$ if there is a directed path from $u$ to $v$. Choose any linear extension of this order and list the vertices (cycles) from right to left accordingly; for instance,
  \begin{equation*}
    (2\,5\,6) \cdot (2\,6) \cdot (1\,7) \cdot (1\,2) \cdot (2\,3)
    \cdot (3\,7) \cdot (3\,4).
  \end{equation*}
Observe that this is indeed a factorization of $(1\,2\,3\,4)(5\,6)(7)$.  
\end{example}

  \begin{figure}[t]
    \centering
    \includegraphics[scale=0.85]{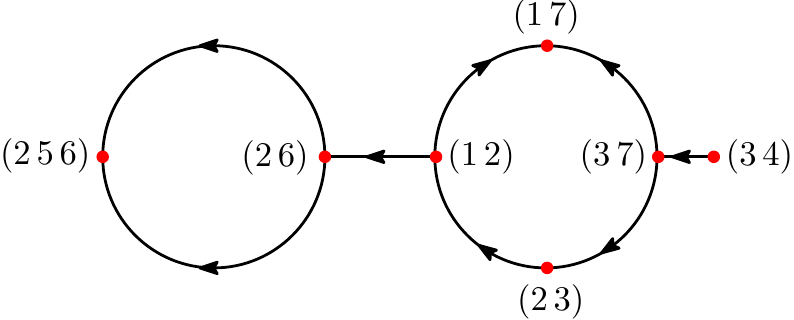}
    \caption{Reconstructing of a factorization from a map: Recovering and ordering the factors.}
    \label{fig:reconstruction2}
  \end{figure}

\subsection{Connections with Mesoscopic Physics}
\label{sec:physics}

Interestingly, very similar maps arise in mesoscopic physics, when
trying to evaluate statistical moments of electron transport through
an irregularly shaped (or ``chaotic'') cavity.  One approach to the
problem approximates the quantum probability of transmission through
the cavity by a sum over all classical trajectories that could be
taken by a billiard ball.  A series of approximations reduces the
problem to the enumeration of pairs of sets of curves possesing the
properties described below.

\begin{figure}[t]
  \centering
  \includegraphics[scale=1]{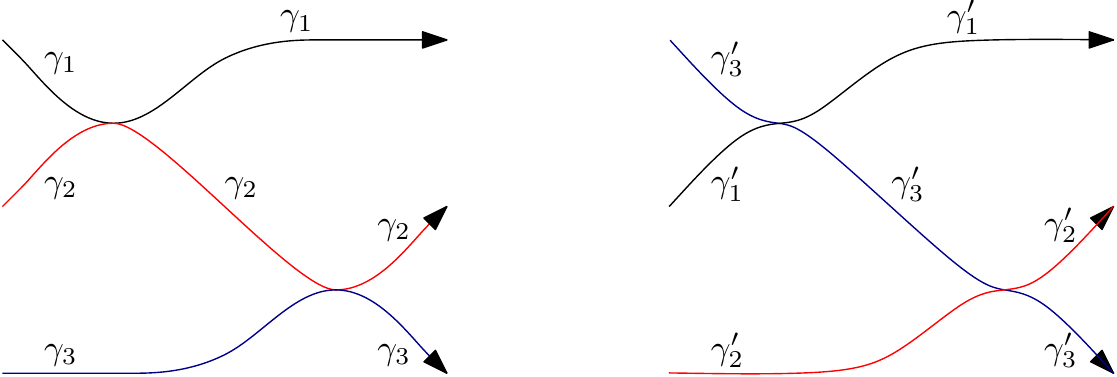}
  \caption{Two sets of curves occupying the same space and connecting
    the same sets of points but in a different order.  The curves on
    the left ``bounce'' off each other, while the curves on the right
    cross over.}
  \label{fig:leadingorder}
\end{figure}

Let $\{\gamma_j(t)\}_{j=1}^n$ and $\{\gamma'_j(t)\}_{j=1}^n$ be the
two sets of curves parametrized by $t \in [0,1]$.  Then
\begin{enumerate}
\item the curves connect the same sets of points but in a different
  order, specified by the permutation $\pi$,
  \begin{equation*}
    \gamma'_j(0) = \gamma_{\pi(j)}(0), \qquad \gamma'_j(1) = \gamma_j(1),
  \end{equation*}
\item the curves in one set occupy the same space as the curves in the
  other,
  \begin{equation*}
    \bigcup_{j} \gamma_j([0,1])
    =
    \bigcup_{j} \gamma'_j([0,1]).
  \end{equation*}
\end{enumerate}
To satisfy both conditions, the curves within each set must intersect,
and it is the topologically inequivalent configurations that are to be
counted.  An example of two sets of three curves satisfying the
conditions with permutation $\pi = (1\, 2\, 3)$ is given in
Fig.~\ref{fig:leadingorder}.  In \cite{BerHarNov08} it was observed
that leading order (in a sense we cannot describe here) configurations
are in one-to-one correspondence with minimal length inequivalent
factorizations of the full cycle $(1\, 2\,\ldots\,n)$.  For example,
the configuration of Fig.~\ref{fig:leadingorder} corresponds to the
factorization $(1\, 2\, 3) = (1\,2) \cdot (2\, 3)$; the reader will
observe that the two sets of curves in the figure can be thought of as
the two ways of going through the shuttle diagram $(1\,2) \cdot (2\,
3)$.

This correspondence extends further.
In fact, it was shown in~\cite{BerKui_jmp13a} that all contributing
configurations (``semiclassical diagrams'') satisfy conditions almost
identical to those of Theorem~\ref{thm:digraphs} (corresp.\
Theorem~\ref{thm:even_degree}), with the only difference being the
absence of ``acyclic'' condition (corresp.\ condition (e) is relaxed
to allow zero odd corners).  The cycle type of the permutation $\pi$
would correspond to the type of the physical quantity considered
(linear vs nonlinear moments), while the order at which a diagram
contributes is determined by the genus of the corresponding map.
Since the result of the physics evaluation \cite{BerKui_jmp13a}
coincides with a prediction obtained by integration over the unitary
group, a rich connection between inequivalent factorizations and
random matrix theory was expected and led us, via
\cite{MatNov_fpsac10} and \cite{Col_imrn03}, to
Theorem~\ref{thm:connections}.

\section{Enumeration of Inequivalent Factorizations}
\label{sec:proof}

In this section we prove Theorem~\ref{thm:Q} by counting all corresponding maps according to  Theorem~\ref{thm:even_degree}. We address cases $m=1,2$, and $3$ separately, although  the general method in  each case is the same.  First we classify planar maps with $m$ faces according to their ``backbone structure''. We then generate all applicable maps by planting trees on these structures, using generating series to keep track of vertex degrees and the number of leaves in every face.  Finally, we apply an algebraic filter to ensure the number of odd vertices with respect to every cycle is even and nonzero. 

The \emph{backbone} of a map is the map obtained by recursively removing all leaves. Reducing a graph to its backbone should be viewed as ``removing rooted trees'', where we work with the convention that trees are always rooted at a leaf.  This process is illustrated in Fig.~\ref{fig:backbone}.  
\begin{figure}[t]
  \centering
  \includegraphics[scale=.5]{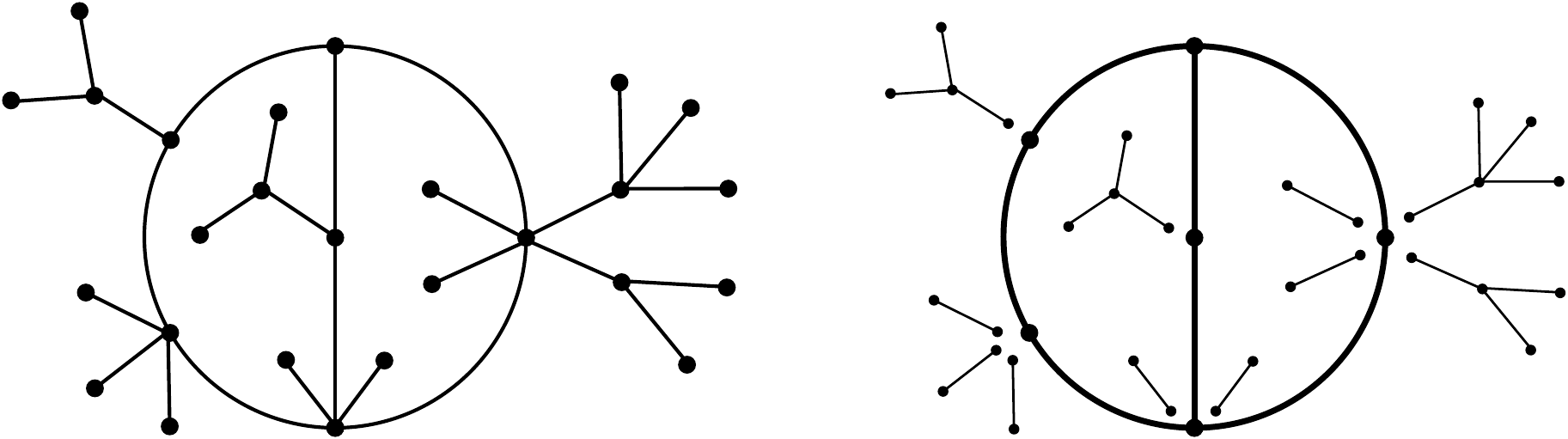}
  \caption{A graph (left) and its decomposition into its
    backbone and rooted trees (right).}
  \label{fig:backbone}
\end{figure}

The \emph{backbone structure} of a map is obtained from its backbone by iteratively removing vertices of degree 2 and merging their incident edges.  Note that this preserves the number of faces of the map.  For planar maps with  one or two faces, the backbone structures are degenerate, involving a single
vertex and a loop, respectively.  There are three possible planar backbone structures
with three faces, as depicted in Fig.~\ref{fig:structures123}.  The
backbone structure of the map in Fig.~\ref{fig:backbone} is the last
structure of Fig.~\ref{fig:structures123}.
\begin{figure}[t]
  \includegraphics[width=0.9\textwidth]{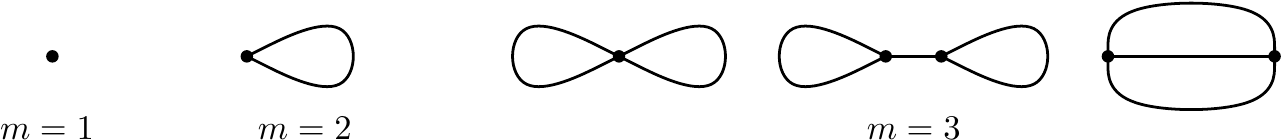}
  \caption{The possible planar backbone structures with $m=1,2$, and $3$ faces.}
  \label{fig:structures123}
\end{figure}

\subsection{Factorizations of Full Cycles ($m=1$)}

\newcommand{\lf}{s}

By Theorem~\ref{thm:even_degree}, factorizations of the full cycle
$(1\ 2\ \cdots n)$ are in one-to-one correspondence with planted plane trees  having
$2n$ leaves  in which every internal vertex has even degree $\geq 4$. 

Let $\xi(\lf, \bq)$ be the generating function for such trees, with
$\lf$ marking non-root leaves and $q_k$ counting vertices of degree $2k$. That is,
\begin{equation}
\label{eq:xi}
	\xi(\lf,\bq) = \sum \icfacts{(n)}{\b} \ \lf^{2n-1} q_2^{\b_2} q_3^{\b_3} \cdots,
\end{equation}
where the sum extends over all $n \geq 1$ and all lists $\b=(\b_2,\b_3,\ldots)$ of nonnegative integers.

The vertex $v$ incident with the root of a tree is either a leaf or a vertex of degree $2k$, for some $k \geq 2$. In the latter case,
removal of $v$ decomposes the tree into a list of $2k-1$ trees. This leads to the recursive relation
\begin{equation*}
  \xi = \lf + q_2 \xi^3 + q_3 \xi^5 + q_5 \xi^7 + \cdots,
\end{equation*}
or equivalently
\begin{equation}
\label{eq:recur_xi}
	\lf=\xi(1-\QQ(\xi^2)),
\end{equation}  
where $\QQ$ is defined as in~\eqref{eq:Q}.
Squaring~\eqref{eq:recur_xi} and comparing with~\eqref{eq:Q_xis} gives
\begin{equation}
  \label{eq:xi_to_xis}
  \xi^2 = \xis(\lf^2),
\end{equation}

To obtain case $m=1$ of Theorem~\ref{thm:Q}, we note
that \eqref{eq:icGS}, \eqref{eq:xi}, \eqref{eq:recur_xi}, and \eqref{eq:xi_to_xis} give
\begin{equation}
\label{eq:prem1}
	\lf^2 \pd{\icGS{1}}{x}\left(\lf^2\right) 
    	= \lf \xi = \xi^2(1-\QQ(\xi^2))
        =\xis(1-\QQ(\xis)).
\end{equation}
Replacing $\lf^2$ with $x_1$ in~\eqref{eq:prem1} now yields~\eqref{eq:Qm1}.
 

\subsection{Factorizations of Permutation with Two Cycles ($m=2$)}

Theorem~\ref{thm:even_degree} associates factorizations of class $(\a_1,\a_2)$ with certain two-faced planar maps. Every such map results from rooting trees on the vertices of a single cycle, and we will generate and count them in exactly this way.

Consider the generating series
\begin{equation}
\label{eq:Phi2}
	\Phi_2(\mathbf{\lf},\bq,\delta)
	= \sum_{\mathcal{M}} \frac{\lf_1^{2\a_1}}{2\a_1} \lf_2^{2\a_2} \delta^c q_2^{\b_2} q_3^{\b_3}\cdots,
\end{equation}
where the sum runs over all two-face planar maps $\mathcal{M}$ satisfying conditions (a) through (d) of Theorem~\ref{thm:even_degree}, and $\lf_i$
marks leaves in face $i$, $q_k$ marks vertices of degree
$2k$, and $\delta$ marks odd corners inside face 1.   We need not track odd corners in both faces of  2-face maps because a vertex is at an odd corner of one face if and only if it is at an odd corner of the other.

We wish to apply Theorem~\ref{thm:even_degree} to express $\icGS{2}$ in terms of $\Phi_2$. However, to enforce condition (e) of the theorem we must first remove all terms of $\Phi_2$ that are either of odd
or zero degree in $\delta$.  This filtration is accomplished  by regarding $\Phi_2$ as a power series over $\Q[\delta]$ and letting the operator
$\map{\Delta}{\Q[\delta]}{\Q}$ defined by
\begin{equation}
\label{eq:Delta}
	\Delta f(\delta) := \frac{f(\delta)+f(-\delta)}{2} \Bigg|^{\delta=1}_{\delta=0}	
			 = \tfrac{1}{2}f(1) + \tfrac{1}{2}f(-1)-f(0).
\end{equation} 
act on its coefficients.  Upon comparing~\eqref{eq:icGS} and~\eqref{eq:Phi2}, we then have
\begin{equation}
  \label{eq:phi_psi}
  \icGS{2}(\mathbf{\lf}^2,\bq) = 2\, \Delta \Phi_2(\mathbf{\lf},\bq,\delta)
\end{equation}

Let us now determine $\Phi_2$ by constructing all relevant maps.  We begin with a cycle $C$ embedded in the plane. Since the inner and
outer faces of this map are interchangeable, we may assume they have
labels 1 and 2, respectively.  To account for the circular symmetry of $C$, we shall fix one of its vertices and plant trees on the resulting rooted cycle $C'$.  

Let $\xi$ be defined as before (see~\eqref{eq:xi}), so that trees planted on $C'$ in face $i$ are recorded by  $\xi_i=\xi(x_i,\bq)$. Since the vertices of $C'$ begin with degree 2 and must have even degree $\geq 4$ after all trees are planted, a positive even number of trees must be planted at each.  If $C'$ has $k$ vertices, then all maps that can arise in this way are generated by
$
	(\ve+\delta\vo)^k,
$
where
\begin{align}
  \label{eq:vertex_2_even}
  \ve &= q_2(\xi_1^2 + \xi_2^2) 
  + q_3(\xi_1^4 + \xi_1^2\xi_2^2 + \xi_2^4) 
  + q_4(\xi_1^6 + \xi_1^4\xi_2^2 + \xi_1^2\xi_2^4 + \xi_2^6) 
  + \ldots \\
     &= \frac{\xi_1^2 Q(\xi_1^2) - \xi_2^2 Q(\xi_2^2)}{\xi_1^2 - \xi_2^2} \nonumber
\end{align}
accounts for plantings that result in an even corners inside $C'$, and
\begin{align}  
  \label{eq:vertex_2_odd}
  \vo &= q_2\xi_1 \xi_2 
  + q_3(\xi_1^3\xi_2 + \xi_1 \xi_2^3) 
  + q_4(\xi_1^5\xi_2 + \xi_1^3\xi_2^3 + \xi_1 \xi_2^5) 
  + \cdots \\
  &= \frac{\xi_1\xi_2 \big(Q(\xi_1^2) - Q(\xi_2^2)\big)}{\xi_1^2 - \xi_2^2} \nonumber
\end{align}
accounts for plantings that result in an odd corner.    Since the resulting maps have no symmetries, a leaf in face 1 can be distinguished arbitrarily. Said differently, each map with a distinguished leaf is generated  with weight $1/\ell$, where $\ell$ is the total number of leaves in face 1, in accordance with~\eqref{eq:Phi2}.
Since $C$ can have any number of vertices, we have
\begin{align}
\label{eq:phi2log}
  \Phi_2(\mathbf{\lf},\bq,\delta) 
  = \sum_{k \geq 1} \frac{1}{k} (\ve+\delta\vo)^k 
  = \log(1 - \ve - \delta\vo)^{-1},
\end{align}
where the factor $1/k$ is present to undo the rooting of  $C'$. 

From ~\eqref{eq:phi_psi}--\eqref{eq:phi2log} we obtain
\begin{align*}
\label{eq:prelim_2}
	\icGS{2}(\mathbf{\lf}^2,\bq) 
	&= \Phi_2(\mathbf{\lf},\bq,1)+\Phi_2(\mathbf{\lf},\bq,-1)-2\Phi_2(\mathbf{\lf},\bq,0)  \\
	&= \log\pr{\frac{(1-\ve)^2}{(1-(\ve+\vo))(1-(\ve-\vo))}} \\
	&= \log\pr{\frac{\pr{\xi_1^2(1-\QQ(\xi_1^2))-\xi_2^2(1-\QQ(\xi_2^2))}^2}	{(\xi_1^2-\xi_2^2)(\xi_1^2(1-\QQ(\xi_1^2))^2-\xi_2^2(1-\QQ(\xi_2^2))^2)}}.
\end{align*}
Finally,~\eqref{eq:recur_xi}, \eqref{eq:xi_to_xis} and the
substitution $\mathbf{\lf}^2 = \mathbf{x}$ yield
$$
	\icGS{2}(\mathbf{x},\bq) = 
	\log\pr{\frac{\pr{\xis_1\QS_1-\xis_2\QS_2}^2}{(\xis_1-\xis_2)(x_1-x_2)}},
$$
where $\xis$ and $S$ are defined as in Theorem~\ref{thm:Q}.   It is now routine to verify~\eqref{eq:Qm2} by differentiating the above expression and
simplifying with~\eqref{eq:xdx}.

We remark in closing that the filtration $\Delta$ does not appear in the earlier derivation~\cite{Irv_cjm09} of this same result. The graphical model of inequivalent factorizations employed there was in some sense dual to the one used here, and the analogue of condition (e) of Theorem~\ref{thm:even_degree} was hidden.  This is precisely what hindered earlier efforts to derive $\icGS{3}(\mathbf{x},\bq)$ via that model.

\subsection{Factorizations of Permutations with Three Cycles ($m=3$)}
\label{sec:factor_m3}

The three distinct backbone structures for three-face planar maps are 
shown in Fig.~\ref{fig:structures3} along with their symmetry groups.
(Note that we consider only orientation-preserving symmetries.)
\begin{figure}[t]
  \includegraphics[width=0.7\textwidth]{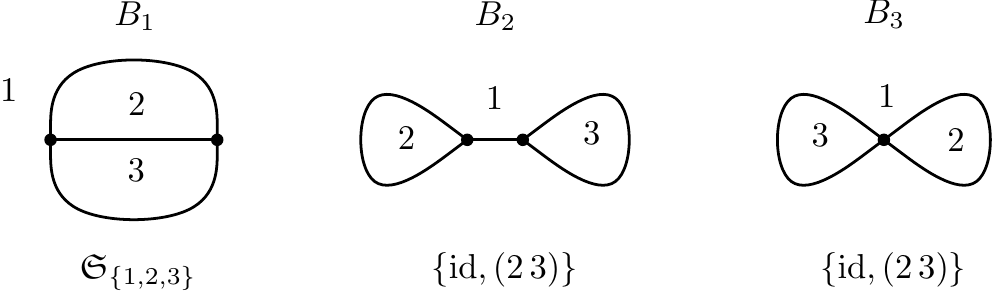}
  \caption{Planar backbone structures with three faces, along with their symmetry groups.}
  \label{fig:structures3}
\end{figure}
We shall refer to these structures as $B_1$, $B_2$, and $B_3$, as indicated in the figure.

Every factorization of a permutation with three cycles corresponds to
a map obtained by planting trees on the edges and vertices of some $B_j$.
We will generate all relevant maps in this way, initially recording the number of odd corners with respect to all boundary walks and then filtering the results to enforce
condition (e) of Theorem~\ref{thm:even_degree}.  As per our remarks following the Theorem, we will in fact generate maps \emph{without} a distinguished leaf and rely on the
$2\alpha_1\alpha_2\alpha_3$-to-one correspondence between factorizations and such maps.

\newcommand{\partialGS}[1]{\Phi_{3,#1}} 

To this end, let $\partialGS{j}$ be the generating
series for maps having backbone structure $B_j$ and satisfying conditions (a), (b), and (d) of Theorem~\ref{thm:even_degree}. As before, leaves in face $i$ will be marked by $\lf_i$, vertices of degree $2k$ by $q_k$, and $\delta_i$ will mark odd corners along the boundary of face $i$
of the backbone.

\newcommand{\edge}[2]{\epsilon(#1,#2)}
An edge $e$ of a backbone structure can support any number of vertices, each of which is at a corner of the boundary walks of the faces separated by $e$.  The contribution of any such edge bordering faces $a$ and $b$ (which may be the same) is therefore
\begin{equation}
  \label{eq:edge_contrib}
  \edge{a}{b} = \frac{1}{1-\ve(a,b)-\delta_a\delta_b \vo(a,b)},
\end{equation}
where 
\begin{align}
\label{eq:vertex_h_sums}
	\ve(a,b) = \sum_{j \geq 2} q_{j} h_{j-1}(\xi_a^2, \xi_b^2), \qquad
	\vo(a,b) = \xi_a \xi_b \sum_{j \geq 2} q_{2} h_{j-2}( \xi_a^2, \xi_b^2).
\end{align}
Note that these expressions are identical to~\eqref{eq:vertex_2_even} and~\eqref{eq:vertex_2_odd}.  

Sums of this form will arise repeatedly in what follows, and  Lemma~\ref{lem:umbral} (in the appendix) shows how they are resolved into rational expressions involving $\QQ$ and its derivatives.  For example, applying Lemma~\ref{lem:umbral} to~\eqref{eq:vertex_h_sums} and~\eqref{eq:edge_contrib} results in
\begin{equation*}
\label{eq:full_edge_contrib}
	\edge{a}{b} =
	\begin{cases}
		\frac{\xi_a^2-\xi_b^2}{\xi_a^2(1-Q_a)-\xi_b^2(1-Q_b)-\delta_a\delta_b\xi_a\xi_b(Q_a-Q_b) }
		&\text{if $a \neq b$} \\
		\frac{1}{1-Q_a-\xi_a^2Q'_a(1+\delta_a^2)}
		&\text{if $a=b$},
	\end{cases}
\end{equation*}
where $Q_i=Q(\xi_i^2)$ and $Q'_i=Q'(\xi_i^2)$.  Notice that when $\delta_a=1$ we get $\edge{a}{a} = 1/\QD(\xi_a^2)$, where $\QD$ is defined as in Theorem~\ref{thm:Q}.

Consider backbone structure $B_1$, which contains two vertices of degree three, each incident with three faces. To determine $\partialGS{1}$, we first find the contribution $\nu_3(a,b,c)$ of a general vertex of degree 3 incident with faces $a$, $b$
and $c$.  Such a vertex is required to have even degree once all trees are planted, so an odd number of trees must be planted in exactly one or three of its incident faces.  Therefore
\begin{equation}
  \label{eq:vertex3}
	\begin{split}
  \nu_3(a,b,c) &=
	e_1(\delta_a \xi_a, \delta_b \xi_b, \delta_c \xi_c) 
  \sum_{j \geq 2} q_{j} h_{j-2}(\xi_a^2,\xi_b^2,\xi_c^2) \\
  &\qquad\qquad + 
  	e_3(\delta_a \xi_a, \delta_b \xi_b, \delta_c \xi_c) 
  \sum_{j \geq 3} q_{j} h_{j-3}(\xi_a^2,\xi_b^2,\xi_c^2).
  \end{split}
\end{equation}
Taking edges into account, we have
\begin{align*}
\label{eq:partial1}
  \partialGS{1} &= \nu_3(1,2,3)\nu_3(1,3,2) \edge{1}{2}\edge{2}{3}\edge{3}{1}.
\end{align*}
Note that face labels have been assigned in only one way because of the full
symmetry group.  

Now consider structure $B_2$. We regard its two vertices of degree 3 as 
being incident with three faces, \emph{two of which are identical}, so that their analysis
is identical to that above.  Their contributions to $\partialGS{2}$
are therefore $\nu_3(1,2,1)$ and $\nu_3(1,3,1)$, where $\nu_3$ is
given by~\eqref{eq:vertex3}.  Since $B_2$ has two
symmetries, its faces may be labelled in $3!/2=3$ distinct ways.
But rather than summing over the three distinct label assignments we sum over
\emph{all} labellings and divide by the size of the symmetry group.
This gives
\begin{equation*}
\label{eq:partial2}
	\partialGS{2} = \frac{1}{2} \sum_{\{a,b,c\} = \{1,2,3\}} 
        \nu_3(a,b,a)\nu_3(a,c,a) \edge{a}{b} \edge{a}{c} \edge{a}{a},
\end{equation*}
where the summation is over all permutations of $\{1,2,3\}$.

Finally, we consider structure $B_3$, which contains only a vertex of degree 4.  Despite this vertex being incident with only 3 distinct faces, we again consider a general vertex of degree 4 incident with faces $a$, $b$, $c$ and $d$.  For such a vertex to
remain of even degree, an odd number of trees must be planted in 0, 2,
or 4 of its corners, and an even number in the rest.  The vertex contribution is therefore
\begin{multline*}
  \label{eq:vertex4}
     \nu_4(a,b,c,d) 
     =
     \sum_{j \geq 2} q_{j} h_{j-2}(\xi_a^2,\xi_b^2,\xi_c^2,\xi_d^2)  \\
     + e_2(\delta_a\xi_a, \delta_b\xi_b, \delta_c\xi_c, \delta_d\xi_d)
     \sum_{j \geq 3} q_{j} h_{j-3}(\xi_a^2,\xi_b^2,\xi_c^2,\xi_d^2)  \\
     + e_4(\delta_a\xi_a, \delta_b\xi_b, \delta_c\xi_c, \delta_d\xi_d)
     \sum_{j \geq 4} q_{j} h_{j-4}(\xi_a^2,\xi_b^2,\xi_b^2,\xi_d^2),
\end{multline*}
With the faces of $B_3$ labelled as in Fig.~\ref{fig:structures3}, it contributes $\nu_4(1,2,1,3) \edge{1}{2} \edge{1}{3}$ to $\partialGS{3}$. Accounting for symmetries gives
\begin{equation*}
\label{eq:partial3}
  \partialGS{3} = \frac{1}{2} \sum_{\{a,b,c\} = \{1,2,3\}} \nu_4(a,b,a,c) \edge{a}{b} \edge{a}{c}.
\end{equation*}

We now sum the contributions of $B_1$, $B_2$, and $B_3$, and filter the results by applying the operator $\Delta$ (defined
in~\eqref{eq:Delta}) separately to $\delta_1, \delta_2$, and
$\delta_3$.  Comparing the definitions of
$\icGS{3}$ and $\partialGS{}$ we have
\begin{equation}
\label{eq:final_result}
	\icGS{3}(\mathbf{\lf}^2,\bq) = 2 \Delta_1 \Delta_2 \Delta_3 ( \partialGS{1}+ \partialGS{2}+ \partialGS{3}).
\end{equation}
where $\Delta_i$ is the $\Delta$ operator with respect to $\delta_i$.   
Note that this is the correct filtration because every cycle in a
three-face planar map is itself the boundary of some face of the
backbone. (This is \emph{not} true for maps with more than 3
faces. See Figure~\ref{fig:additional_cycles} in Section~\ref{sec:transpositions}.)

The computations are clearly too elaborate to carry out by hand, so we have relied extensively on Maple to evaluate and simplify~\eqref{eq:final_result}.   With considerable human guidance, Maple confirms the result to be the surprisingly compact expression~\eqref{eq:Qm3}.  

We cannot yet satisfactorily explain the dramatic simplification of~\eqref{eq:final_result}.  None of the individual backbone contributions $\partialGS{1},\partialGS{2}$, or $\partialGS{3}$ simplify in any appreciable way (either before or after filtration).  Given the connections surveyed in Sections~\ref{sec:survey_connections} and~\ref{sec:survey_cyclefacts}, we interpret this global simplification as strong evidence that transitive factorizations (in general) possess a rich  unknown structure.  




\section{Enumeration of Inequivalent $k$-cycle Factorizations}
\label{sec:transpositions}

In this section we consider the restriction of Theorem~\ref{thm:Q} to $k$-cycle factorizations and, even more specifically, to factorizations into transpositions.   In the latter case we shall also describe how a specialization of our main graphical correspondence has been used to count factorizations of permutations containing four cycles.

\subsection{Specializations}

Theorem~\ref{thm:Q} is readily specialized to obtain generating series for inequivalent $k$-cycle factorizations of permutations with up to three cycles. Upon substituting $q_k=1$ and $q_i =0$ for $i \neq k$ throughout the theorem we arrive at the following result.  Simplification to the forms below is straightforward in cases $m=1$ and $m=2$, but relies on Lemmas~\ref{lem:3case_simplification} and~\ref{lem:det} of the Appendix in case $m=3$.  

\begin{cor}
\label{cor:kcycles}
Let $\xis \in \Q[[x]]$ be the unique solution of
$
	\xis = x(1-\xis^{k-1})^{-2},
$
namely
\begin{align*}
	\xis(x) = \sum_{i \geq 0} \frac{1}{1+i(k-1)} \binom{1+i(2k-1)}{i} x^{1+i(k-1)},
\end{align*}
and let $\xis_i=\xis_i(x_i)$.  Then
\begin{align*}
	\Dopx \ikGS{1}{k} &= \xis_1(1-\xis_1) \\
	\Dopx \ikGS{2}{k}
	&= \frac{2(k-1) \xis_1 \xis_2 \hsf{k-2}(\xis_1,\xis_2)^2} {\prod_i (1-(2k-1)\xis_i^{k-1}) \pr{1-\hsf{k-1}(\xis_1,\xis_2)} } \\
	\ikGS{3}{k}
	&=
	\frac{2\xis_1 \xis_2 \xis_3 G(G+G')}
	{\prod_{i} (1-(2k-1)\xis_i^{k-1}) \prod_{i < j} (1-\hsf{k-1}(\xis_i,\xis_j))},
\end{align*}
where in the formula for $\ikGS{3}{k}$ we have let 
\begin{align*}
G &= s_{(k-3)} - (2k-1)s_{(k-2)^2} \\
G' &= s_{(k-2)^2}-(2k-1)s_{(2k-3,k-2)}
\end{align*}
with $s_{\l} \equiv s_{\l}(\xis_1,\xis_2,\xis_3)$. \qed
\end{cor}


Corollary~\ref{cor:kcycles} simplifies considerably when further specialized at $k=2$. The result is the following generating series for inequivalent minimal transitive factorizations into transpositions. Although it is not immediately obvious,~\eqref{eq:gouldenjackson} is indeed equivalent to the form of $\ikGS{2}{2}$ given here.
\begin{cor}
\label{cor:transpositions}
Let $\xis \in \Q[[x]]$ be the unique solution of $\xis = x(1-\xis)^{-2}$, namely
\begin{equation*}
	\xis(x) = \sum_{n \geq 1} \frac{1}{n} \binom{3n-2}{n-1} x^n.
\end{equation*}
Letting $\xis_i = \xis(x_i)$, we have
\begin{equation*}
  \label{eq:transposition_gs}
   \begin{split}
	\Dopx \ikGS{1}{2} 
		&= \xis_1(1-\xis_1) \\
	\Dopx \ikGS{2}{2} 
		&= \frac{2 \xis_1 \xis_2} {(1-3\xis_1)(1-3\xis_2)(1-\xis_1-\xis_2)} \\
	\ikGS{3}{2}
  		&= \frac{ 6 \xis_1 \xis_2 \xis_3 (4 - 3\xis_1 - 3\xis_2 - 3\xis_3)}{(1-3\xis_1) 		
  			(1-3\xis_2) (1-3\xis_3) (1-\xis_1-\xis_2) (1-\xis_2-\xis_3) (1-\xis_1-\xis_3)}.
    \end{split}
\end{equation*}\qed
\end{cor}

It is cumbersome to extract coefficients from the series $\ikGS{m}{2}$ in the forms given above.  However, a change of variables makes this task more palatable.  Set $g = \xis/(1-\xis)$, or equivalently $\xis = g/(1+g)$, so that the defining equation $\xis=x(1-\xis)^2$ becomes\footnote{Comparison with~\eqref{eq:longyear} identifies $1+g$ with Longyear's series $h$.}
\begin{equation*}
\label{eq:gseries}
	g = x(1+g)^3.
\end{equation*}
It is then easy to verify that
$$	
		\frac{1}{1-3\xis_i} = \frac{x_i}{g_i} \pd{g_i}{x_i} 
	\qquad\text{and}\qquad
		\frac{\xis_i\xis_j}{1-\xis_i-\xis_j} = \frac{g_i g_j}{1-g_i g_j},
$$
where $\xis_i=\xis(x_i)$ and $g_i = g(x_i)$.
Thus, for instance, Corollary~\ref{cor:transpositions} gives
\begin{align*}
	D_2 \ikGS{2}{2}(\bx) 
	&= 2 x_1 \pd{g_1}{x_1} x_2 \pd{g_2}{x_2}\, \frac{1}{1-g_1 g_2}.
\end{align*}
Lagrange inversion is now readily applied to extract the coefficient of $x_1^n x_2^m$ on the right-hand side. The result is equation~\eqref{eq:coeffextract}, given in the introduction, for the number of inequivalent minimal transitive 2-cycle factorizations of any permutation of cycle type $(n,m)$.  A similar but substantially more complicated expression can be derived for the coefficients of $\ikGS{3}{2}$.

\subsection{Factorizations into Transpositions of Permutations with Four Cycles ($m=4$)}
\label{sec:4case_transposition}
There are several significant difficulties associated with applying
our graphical approach to factorizations of permutations with more than three cycles.  The most immediate obstacle is that the number of distinct backbone structures increases very rapidly with the number of faces.  A more subtle difficulty is that one can no longer guarantee condition (e) of Theorem~\ref{thm:even_degree} simply by verifying it on the boundary walks of faces; indeed,  cycles are  \emph{not} necessarily face boundaries, as is clear in Figure~\ref{fig:additional_cycles}.

However, the situation is somewhat simpler if we restrict our attention to
2-cycle factorizations. The internal vertices of the corresponding maps
are then required to have degree 4, which imposes some simplifying
restrictions on the backbones and the contributions of their
components.  This has allowed us to derive the following expression for $\ikGS{4}{2}$, again by generating all corresponding 4-face planar maps according to Theorem~\ref{thm:even_degree}. 

We now briefly describe our derivation of Theorem~\ref{thm:transpositions}.  The relevant backbone structures are shown in Fig.~\ref{fig:structures4}, along with the generators and sizes of their symmetry groups. 
\begin{figure}[t]
  \centering
  \includegraphics[scale=0.9]{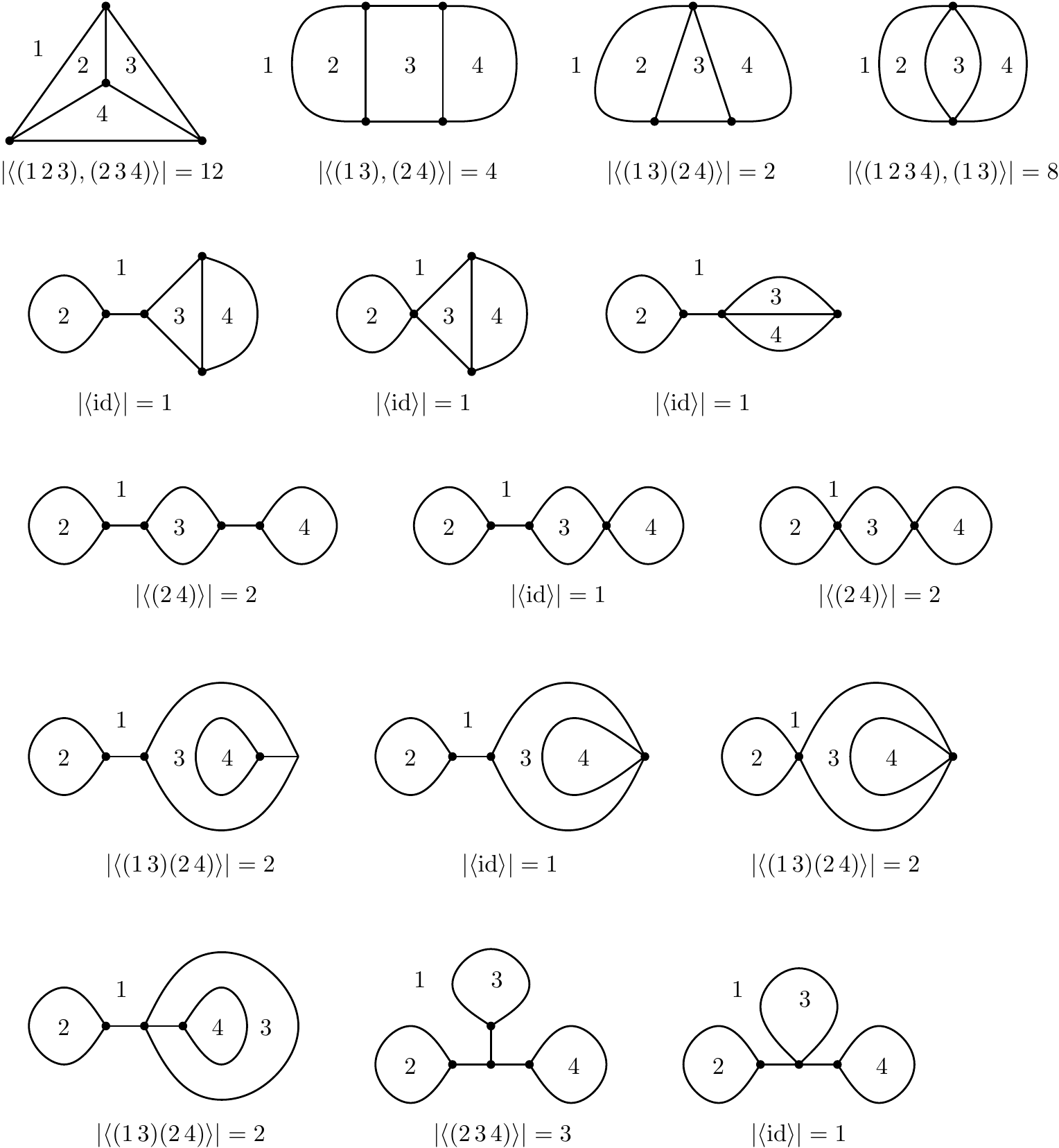}
  \caption{Backbone structures corresponding to 2-cycle factorizations of a permutation with 4 cycles, along with the generators and sizes of their symmetry groups.}
 \label{fig:structures4}
\end{figure}
As mentioned above, we must enforce condition (e) of Theorem~\ref{thm:even_degree} on more than just the boundaries of faces.  For example, Fig.\ref{fig:additional_cycles} shows the extra cycles that must be verified for the first structure of Fig.~\ref{fig:structures4}.
It can be checked that three is the maximal number of  additional
cycles one needs to consider for planar maps with four faces.
\begin{figure}[t]
  \centering
  \includegraphics{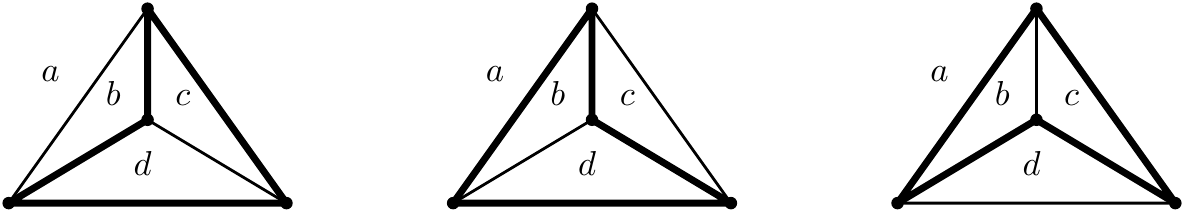}
  \caption{Additional cycles to be checked for compliance with condition (e) of
   Theorem~\ref{thm:even_degree}.  From left to right, we refer to these cycles as
    $ab$, $ac$ and $ad$, since they are sums of the
    boundary walks of the named faces.}
  \label{fig:additional_cycles}
\end{figure}
The contributions of the edges and
vertices of each backbone must 
account for these additional cycles.   

For instance, the contribution of the top left edge of the
structure shown in Fig.~\ref{fig:additional_cycles} is
\begin{equation*}
  \edge{a}{b} = \frac{1}{1-\ve(a,b)-\delta_a\delta_b\delta_{ac}\delta_{ad} \vo(a,b)}.
\end{equation*}
This should be compared with~\eqref{eq:edge_contrib}.  The
four $\delta$ factors arise because the edge lies on the boundaries of 
faces $a$ and $b$ and also on the additional cycles
$ac$ and $ad$.  Furthermore, because each vertex added to
this edge must have degree four, the vertex contributions~\eqref{eq:vertex_h_sums} simplify to
\begin{align*}
  \label{eq:even_vertex4}
  \ve(a,b) = \xi_a^2 + \xi_b^2,\qquad
  \vo(a,b) = \xi_a\xi_b.
\end{align*}
Note that we have suppressed the indeterminate $q_2$ as it is redundant.

The contribution of each vertex of a backbone structure is 
also greatly simplified by the fact that all vertices must end up with degree 4.
Consider, for example, the the top vertex of the
structure in Fig.~\ref{fig:additional_cycles}.  Since this vertex has degree 3, it must support exactly one tree.  This tree can lie in any of the three incident faces, so its contribution is
\begin{equation*}
  \nu_3 = \delta_a\delta_{ad}\xi_a +\delta_b\delta_{ac}\xi_b
  +\delta_c\delta_{ab}\xi_c.
\end{equation*}
The top vertex of the third diagram in Fig.~\ref{fig:structures4} has degree 4, so it cannot support any trees; thus its contribution is
\begin{equation*}
  \nu_4 = \delta_{12} \delta_{14},
\end{equation*}
since it is an odd vertex with respect to the cycles formed by adding
boundary walks of  faces $1$, $2$ and faces $1$, $4$, correspondingly.

In a similar manner we find the contribution of each
diagram in Fig.~\ref{fig:structures4}. Upon summing the results, filtering, and simplifying (with the aid of Maple), we arrive at Theorem~\ref{thm:transpositions}. 
This expression for $\ikGS{4}{2}$ is by far the simplest we have found, but it was only discovered by first conjecturing the general form and then guiding Maple to simplify toward such a result. We therefore caution that it is by no means clear it is a \emph{natural} form for the series.  It is best considered a hard won data point in our attempt to uncover the general structure of inequivalent factorizations. 

In principle, it is possible to formalize our derivation of $\ikGS{4}{2}$ and ``automate'' the computation of $\ikGS{m}{2}$ for $m > 4$.  This process would begin with parametrizing the possible backbone structures, say using Tutte's axiomatization via triples of permutations~\cite{Tutte_GraphTheory,JacVis_Atlas}.  However, our  experience suggests that the benefit would be very limited due to the rapidly increasing complexity (see  also~\cite{BerKui_jmp13b}) and consequent inability to effectively simplify the results. Even simplifying $\ikGS{4}{2}$ to the relatively compact form of Theorem~\ref{thm:transpositions} was a  considerable undertaking.

\section{Acknowledgements}

JI would like to thank Guillaume Chapuy, Dominique Poulalhon, and
Gilles Schaeffer for interesting and helpful discussions during the
preparation of this paper.  The work of GB was supported in part by the NSF DMS Grant 0907968.

\appendix

\section{Technical Lemmas}
\label{sec:technical}

Throughout, we let $V(\bx) = \prod_{i < j} (x_i-x_j)$ be the
Vandermonde in the indeterminates $\bx = (x_1,x_2,\ldots,x_m)$, and we write
$\lVert A \rVert$ for the determinant of a matrix $A$.

\begin{lem}
\label{lem:3case_simplification}
For indeterminates $\mathbf{a}=(a_1,a_2,a_3), \mathbf{b}=(b_1,b_2,b_3)$, and $\mathbf{z}=(z_1,z_2,z_3)$,
\begin{equation*}
	\sum_{i = 1}^3 \frac{1}{z_i} \prod_{j \neq i} \frac{z_i-z_j}{(a_i-a_j)(b_i-b_j)}
	= \frac{1}{z_1 z_2 z_3 V(\mathbf{a})V(\mathbf{b})}
			{\begin{Vmatrix} 
			a_1 z_1 & a_2 z_2 & a_3 z_3 \\
			z_1 & z_2 & z_3 \\
			1 & 1 & 1
			\end{Vmatrix}} 
			\begin{Vmatrix}
			b_1 z_1 & b_2 z_2 & b_3 z_3 \\
			z_1 & z_2 & z_3 \\
			1 & 1 & 1
			\end{Vmatrix}.
\end{equation*}
\end{lem}

\begin{proof}
Direct expansion.
\end{proof}

\begin{lem}
\label{lem:det}
For any positive integers $p > q$ we have
\begin{equation*}
		\frac{1}{V(\bx)}
		\begin{Vmatrix}
			x_1^p & x_2^p & \cdots & x_m^p \\
			x_1^q & x_2^q & \cdots & x_m^q \\
			x_1^{m-3} & x_2^{m-3} & \cdots & x_m^{m-3} \\
			x_1^{m-4} & x_2^{m-4} & \cdots & x_m^{m-4} \\
			\vdots & \vdots & \ddots & \vdots \\
			1 & 1 & \cdots & 1
		\end{Vmatrix}
		=  \ssf{(p+1-m,q+2-m)}(\bx).
\end{equation*}
\end{lem}

\begin{proof}
This is the classical definition of the Schur polynomial.
\end{proof}

\begin{lem}
\label{lem:umbral}
Let $A(t) = \sum_{i \geq d} a_i t^{i-1} \in \C[[t]]$.  For any positive integer $m$, and for any integer $s \geq 1-d$,
we have
\begin{equation*}
	\sum_{i \geq d} a_{i} h_{i-m+s} (x_1,\ldots,x_m) = \sum_{i = 1}^m \frac{x_i^s A(x_i)}{\prod_{j \neq i} (x_i-x_j)}.
\end{equation*}
Moreover, for $s \geq 2-d$ we have the following evaluation at $x_m=x_1$:
\begin{equation*}
	\sum_{i \geq d} a_{i} h_{i-m+s} (x_1,\ldots,x_{m-1},x_1) = 
		\pd{}{x_1} \sum_{i = 1}^{m-1} \frac{x_i^{s} A(x_i)}{\prod_{j \neq i} (x_i-x_j)}.
\end{equation*}
\end{lem}

\begin{proof}
Let $\bx=(x_1,\ldots,x_m)$. 
For every $i \geq d$, let $B_i$ be the $m \times m$ matrix with first row $\bx^{s+i-1}$ and 
with $r$-th row $\bx^{m-r}$,  $r > 1$. Note that the condition $s \geq 1-d$ ensures the entries in the first
row of $B_i$ are polynomial.  Therefore Lemma~\ref{lem:det} gives 
$\det B_i = \hsf{s+i-m}(\bx) V(\bx)$.
Now consider the matrix
\begin{align*}
	B = \begin{bmatrix}
		x_1^s A(x_1) & x_2^s A(x_2) & \cdots & x_m^s A(x_m) \\
		x_1^{m-2} & x_2^{m-2} & \cdots & x_m^{m-2} \\
		x_1^{m-3} & x_2^{m-3} & \cdots & x_m^{m-3} \\
		\vdots & \vdots & \ddots & \vdots \\
		1 & 1 &  \cdots & 1
		\end{bmatrix}.
\end{align*}
Since $x_j^s A(x_j) = \sum_{i \geq d} a_i x_j^{s+i-1}$, we have
\begin{equation*}
	\det B = \sum_{i \geq d} \det B_i = V(\bx) \sum_{i \geq d}  a_i \hsf{s+i-m}(\bx).
\end{equation*}
But expansion along the first row of $B$ gives
\begin{equation*}
	\det B = \sum_{i =1}^m (-1)^{i+1} x_i^s A(x_i) V_i,
\end{equation*}
where  $V_i = \prod_{j < k}^{j,k \neq i} (x_j-x_k)$ is the Vandermonde in the variables $\{x_1,\ldots,x_m\}\setminus\{x_i\}$.  This proves the first statement of the lemma, since $V_i/V(\bx) = (-1)^{i-1} \prod_{j \neq i} (x_i-x_j)^{-1}$.
The second statement follows by noting that
$\hsf{k}(x_1,\ldots,x_{m-1},x_1)=\pd{}{x_1} \hsf{k+1}(x_1,\ldots,x_{m-1})$.  The restriction $s\geq 2-d$ ensures that all expressions are formal power series.
\end{proof}


\bibliographystyle{elsarticle-num}
\bibliography{InequivalentFactorizations}

\end{document}